\documentclass[leqno,final]{siamltex}
\usepackage{graphicx} 
\usepackage{amsmath,amstext,amssymb,bm}
\usepackage{leftidx}

\usepackage{xcolor} 
\usepackage{soul} 
\usepackage{tikz}
\usetikzlibrary{shapes,arrows}
\usepackage{mathrsfs}
\usepackage{epstopdf}
\usepackage{color}
\usepackage{multirow}
\usepackage{tabularx}
\usepackage[shortlabels]{enumitem}

\usepackage{colortbl}
\usepackage{booktabs}

\usepackage{subcaption}

\allowdisplaybreaks[3]  

\numberwithin{equation}{section}
\newtheorem{remark}{Remark}[section]

\allowdisplaybreaks[4]

\usepackage{hyperref}
\hypersetup{
	colorlinks=true,   
	linkcolor=blue,    
	citecolor=blue,    
	filecolor=magenta, 
	urlcolor=cyan      
}

\def\u{\mathbf{u}}

\def\n{\textbf{n}}

\begin{document}
	
	\title{A Thermodynamically Consistent Cahn--Hilliard--Navier--Stokes Model for Tumor Growth\thanks{This work was supported by National Natural Science Foundation of China (No. 12471406), Natural Science Foundation of Guangdong Province of China (No. 2024A1515010294) and the Science
			and Technology Commission of Shanghai Municipality (No. 22JC1400900, 22DZ2229014).} 
		}
	\markboth{A Thermodynamically consistent Cahn–Hilliard Navier--Stokes for Tumor Growth}{CHENYANG LI, HUI YU, PING LIN, HAIBIAO ZHENG}

	\author{Chenyang Li
		\thanks{School of Mathematical Sciences, East China Normal University, Shanghai, 200241, China. \texttt{(52275500026@stu.ecnu.edu.cn)}.}
			\and  Ping Lin
		\thanks{Division of Mathematics, University of Dundee, Dundee, DD1 4HN, UK. \texttt{(P.Lin@dundee.ac.uk)} 
		}
			\and Hui Yu
		\thanks{ \textbf{Corresponding author}. School of Mathematics and Physics, University of Science and Technology Beijing, Beijing 100083, China
			\texttt{(h.yu@ustb.edu.cn)}
		}
		\and  Haibiao Zheng
		\thanks{School of Mathematical Sciences, Ministry of Education Key Laboratory of Mathematics and Engineering Applications, Shanghai Key Laboratory of PMMP,  East China Normal University, Shanghai, 200241, China. \texttt{(hbzheng@math.ecnu.edu.cn)} }
	}

	\maketitle
	
	\begin{abstract} 
This work develops a thermodynamically consistent phase-field model for tumor growth based on the energetic variational framework. The model couples Cahn--Hilliard dynamics for tumor and nutrient transport with the incompressible Navier--Stokes equations.
A first-order time discretization scheme based on the Multiple Scalar Auxiliary Variables (MSAV) approach combined with a pressure-correction strategy is proposed to efficiently handle the nonlinear and coupled structure of the system. The proposed scheme is rigorously proven to be unconditionally energy stable and mass--conservative.
Furthermore, optimal first-order temporal error estimates are established for the tumor phase-field variable, the nutrient concentration, and the fluid velocity.
Numerical experiments are finally presented to demonstrate the effectiveness and robustness of the proposed method, as well as to validate the theoretical convergence rates.
	\end{abstract}
	
	\begin{keywords}
Thermodynamically consistent, Tumor growth, Phase-field model, Navier-Stokes,
Multiple scalar auxiliary variables (MSAV),
Error estimates
	\end{keywords}
	
	\begin{AMS}
		65N12, 
		65N15, 
		65N30 
	\end{AMS}
	
\section{Introduction}	\label{introduction}
%
Cancer is one of the most common malignant tumors and remains a leading cause of global morbidity and mortality \cite{ferlay2018}. A major obstacle to its effective detection and treatment lies in the intrinsic complexity of the biological and biochemical mechanisms that govern tumor growth and progression. 
Malignant tumors are characterized by their capacity for local tissue invasion and distant metastasis \cite{enderling2014}. Analyzing the mechanisms that drive this aggressive growth is fundamental to understanding their pathological behavior and clinical fate.


Existing mathematical models of tumor growth are commonly classified according to their treatment of the interface into two main categories: sharp-interface models \cite{roose2007} and diffuse-interface (phase-field) models \cite{xu2020}. While sharp-interface approaches explicitly track the tumor boundary, diffuse-interface methods have emerged as a particularly powerful and flexible framework for modeling complex tumor morphologies and dynamically evolving interfaces. A key advantage of phase-field models lies in their thermodynamically consistent formulation, which naturally accommodates topological changes such as interface merging and splitting without additional geometric reconstruction.
In phase-field models, the tumor boundary is implicitly represented by an order parameter that varies smoothly across a narrow interfacial region, exhibiting a steep but continuous transition between distinct phases. Owing to their favorable analytical and numerical properties—most notably the existence of an energy dissipation law—phase-field approaches have been extensively applied to tumor growth modeling for more than two decades \cite{ebenbeck2019,jiang2015,garcke2018,oden2010,hilhorst2015}.

Several advances have been published regarding the existence and regularity of solutions.
The existence and uniqueness of both weak and strong solutions, along with the existence of a global attractor, are established in \cite{garcke2020}. The singular limit of the problem, which converges to the corresponding sharp-interface model, is characterized in \cite{hilhorst2015}, with a rigorous proof provided recently in \cite{riva2025}. Further related results can be found in \cite{colli2015,colli20152,colli2017,frigeri2015}.

Significant advances have been made in the numerical approximation of tumor growth, both in temporal and spatial discretizations.
\cite{hawkins2012} developed a thermodynamically consistent four-species tumor model within mixture theory, which directly incorporates nutrients to ensure dissipative gradient flow, and presents an energy-stable, mass-conserving numerical scheme that captures complex tumor evolution dynamics. 
\cite{wu2014} proposes a novel second-order, energy-stable numerical scheme for diffuse-interface tumor models, which combines a Crank–Nicolson framework with convex-concave splitting and artificial-diffusivity stabilization to efficiently handle nonlinearities while preserving accuracy and unconditional stability.
The fully-decoupled discontinuous Galerkin method, proposed by  \cite{zou2022} for the Cahn–Hilliard–Brinkman–Ohta–Kawasaki model, enables efficient simulations of complex tumor growth through its novel decoupling strategies. 
While a first-order Fourier-spectral scheme employing the scalar auxiliary variable (SAV) approach is presented in \cite{shen2023}, an error analysis for the scheme is not provided.
\cite{lin2025} proposes structure-preserving Fourier-spectral/SAV schemes under periodic boundary conditions and provides an error analysis for the first-order time semi-discrete scheme that proves optimal temporal convergence.

However, it is noteworthy that many existing models either neglect the interstitial fluid pressure \cite{ribba2006} or treat it as a fixed function or constant \cite{moha2019}, despite physiological evidence highlighting its significant role in tumor prognosis, treatment response, metastasis, and drug delivery \cite{netti2000}. This motivates the development of a more comprehensive mathematical framework that incorporates such critical details--particularly fluid dynamics--to better describe, understand, and predict tumor evolution.
This paper develops a Thermodynamically consistent Cahn--Hilliard Navier--Stokes models for Tumor Growth, which directly incorporates nutrients to ensure dissipative gradient flow, and presents an energy-stable, mass-conserving numerical scheme that captures complex tumor evolution dynamics.

The discretization difficulties for thermodynamically consistent Cahn–Hilliard Navier–Stokes models for Tumor Growth are particularly acute and stem from several intrinsic features:
\begin{itemize}
	\item The Cahn–Hilliard component involves fourth-order spatial derivatives, which induce strong numerical stiffness and typically require mixed formulations or specially designed discretization techniques to ensure stability and accuracy.
	\item The full model constitutes a strongly coupled, nonlinear multiphysics system, in which the phase-field variables, fluid flow, and nutrient transport are tightly coupled through nonlinear source terms, convective interactions, and chemical potential forces, posing significant challenges for the design of efficient and robust numerical schemes.
	\item From a physical and mathematical perspective, it is essential to preserve mass conservation and the thermodynamic energy-dissipation structure at the discrete level; failure to do so may result in nonphysical oscillations or spurious long-time behavior.
	\item Moreover, the presence of nonlinear convective terms and the incompressibility constraint in the Navier–Stokes equations introduces additional challenges, including velocity–pressure coupling that must satisfy appropriate inf–sup stability conditions.
\end{itemize}

Taken together, these challenges strongly motivate the development of novel discretization strategies that achieve a delicate balance between structure preservation (mass and energy), computational efficiency (linear solvers), and rigorous stability and convergence analysis.

%

To tackle these issues, we make the following key contributions:
\begin{itemize}
	\item Derivation of a thermodynamically consistent Cahn–Hilliard–Navier–Stokes tumor growth model using the energetic variational approach (EnVarA).
	\item Development of a structure-preserving numerical scheme by extending the multiple scalar auxiliary variable (MSAV) technique to both the Cahn–Hilliard dynamics and nonlinear convection.
	\item Efficient handling of incompressibility via a projection–correction method, decoupling velocity and pressure.
	\item Design of a first-order time-stepping scheme that is mass-conservative at each step.
	\item Proof of unconditional energy stability via a discrete energy dissipation law that mirrors the continuous thermodynamics.
	\item Establishment of optimal-order convergence under standard regularity assumptions.
\end{itemize}

To the best of our knowledge, this work provides the first thermodynamically consistent Cahn--Hilliard--Navier--Stokes framework for tumor growth that simultaneously incorporates conservative nutrient transport, chemotactic coupling, energetic mass exchange, and the hydrodynamic forces generated by both
chemical potentials.
%

The remainder of this paper is organized as follows. In Section 2, we derive a thermodynamically consistent Cahn–Hilliard–Navier–Stokes model for tumor growth based on the energetic variational framework. Section 3 introduces the required notation and presents the first-order MSAV scheme. Its mass conservation and energy stability are then established.
Section 4 is devoted to the derivation of optimal error estimates for the proposed scheme. In Section 5, we present numerical experiments to demonstrate the accuracy and efficiency of the method. Finally, Section 6 concludes the paper with some remarks and perspectives for future work.

\section{A Thermodynamically Consistent Cahn--Hilliard--Navier--Stokes Model for Tumor Growth}

In this section, we derive the governing equations for the coupled tumor--nutrient--fluid system within the framework of the Energetic Variational Approach (EnVarA), following the methodology developed in \cite{lin2022,lin2025}. Let $\Omega \subset \mathbb{R}^{2}$ be a smooth bounded domain with boundary $\partial\Omega$. We consider a diffuse-interface description of tumor growth in an incompressible viscous fluid. The phase-field variable $\eta$ represents the local volume fraction of tumor cells and characterizes the evolution of the tumor interface. Physically, $\eta$ is interpreted as a tumor volume fraction, with values near one and zero representing tumor and healthy tissue, respectively.
Let $c$ denote the nutrient concentration available for tumor proliferation, such as oxygen or glucose supplied by the surrounding extracellular environment. The fluid velocity is denoted by $\u$, which transports both tumor cells and nutrients through advection. Based on the conservation laws for mass and momentum, the governing kinematic equations are given by
\begin{subequations}\label{tumor-system-1}
	\begin{align}
		\frac{D \eta}{Dt} +\nabla \cdot \mathbb{T}_\eta = \mathbb{S}_\eta,\label{tumor-1}\\
		\frac{D c}{Dt} +\nabla \cdot \mathbb{T}_c = \mathbb{S}_c,\label{tumor-2}\\
		\rho \frac{D\u}{Dt}=\nabla \cdot  \mathbb{T}_\u+\nabla \cdot \sigma_\u,\label{tumor-3}\\
		\nabla \cdot \u =0.\label{tumor-4}
	\end{align}
\end{subequations}
Here, $\mathbb{T}_{\eta}$ and $\mathbb{T}_{c}$ denote the mass fluxes associated with the phase-field variable and the nutrient concentration, respectively. The tensor $\mathbb{T}_{\u}$ represents the viscous stress, while $\sigma_{\u}$ denotes the stress induced by the diffuse interface. The source terms $\mathbb{S}_{\eta}$ and $\mathbb{S}_{c}$ describe nonlinear biochemical reactions, including tumor cell proliferation, apoptosis, nutrient consumption, and other relevant biological processes. In \eqref{tumor-3}, $\rho$ is the constant fluid density and $\u$ denotes the fluid velocity. The incompressibility constraint is imposed through \eqref{tumor-4}. Throughout this work, the material derivative is defined by
$
\frac{Df}{Dt}
=
\frac{\partial f}{\partial t}
+
(\u\cdot\nabla)f.
$
To close the system, we prescribe the following boundary conditions:
$
\mathbb{T}_\eta\cdot \n \big|_{\partial\Omega} = 0,\,
\mathbb{T}_c \cdot \n \big|_{\partial\Omega} = 0,\,
\u \big|_{\partial\Omega}=0.
$

Next, we determine the unknown quantities $\mathbb{T}_\eta$, $\mathbb{T}_c$, $\mathbb{T}_\u$, $\mathbb{S}_\eta,\mathbb{S}_c,\sigma_\u$ by employing the Energetic Variational Approach (EnVarA). 
The total energy of the system consists of the kinetic energy of the fluid, the phase-field mixing energy, the nutrient free energy, and the chemotactic coupling energy \cite{hawkins2012}. These energy components collectively describe the fluid motion, tumor-interface dynamics, nutrient distribution, and chemotactic interactions between tumor cells and nutrients.
\begin{align}\label{tumor-energy-definition}
	E(\eta,c,\u) =& \int_{\Omega} \frac{\rho}{2} |\mathbf{u}|^2 \, dx+\int_{\Omega} \left( \frac{\epsilon^2}{2} |\nabla \eta|^2 + \frac{1}{2\delta} c^2 + g(\eta) + \chi(\eta,c) \right) dx\\
	=&\int_{\Omega} \frac{\rho}{2} |\mathbf{u}|^2 \, dx+\int_{\Omega} \left( \frac{\lambda}{2} \eta^2 + \frac{\epsilon^2}{2} |\nabla \eta|^2 + \frac{1}{2\delta} c^2 + f(\eta) + \chi(\eta,c) \right) dx\notag\\
	:=&E_{fluid}+E_{tumor}.\notag
\end{align}
here, $\epsilon$ is a small parameter related to the thickness of interfacial layers. 
The function  $g(\eta)=\kappa \eta^2(1-\eta)^2$ is the nonlinear free energy density, and $
f(\eta) = g(\eta) - \frac{\lambda}{2} \eta^2,
$
where $\lambda > 0$. The function $\chi(\eta,c)$ represents the chemotaxis energy, and $\delta > 0$ is a small parameter that governs the relative interaction strength between the cells and nutrient species.
\begin{remark}
Note that the term involving $\lambda$ vanishes in the free energy \eqref{tumor-energy-definition}. Consequently, $\lambda$ may be chosen arbitrarily. In the numerical implementation, we set $\lambda = 4\delta \chi_0^2$ to enhance numerical stability. 
This choice is not strictly necessary in practice, and we do not attempt to determine the optimal value of $\lambda$ in the present work. 
A similar stabilization strategy has also been employed in \cite{shenjie2010}.
\end{remark}

With the total energy functional of the system specified, the chemical potentials 
$\mu_\eta$ and $\mu_c$ are obtained via variational differentiation, namely,
\begin{align}
	\mu_\eta &= \frac{\delta E}{\delta \eta} 
	= \lambda \eta - \epsilon^{2} \Delta \eta + f'(\eta)+ \frac{\partial \chi}{\partial \eta}(\eta,c), \\
	\mu_c &= \frac{\delta E}{\delta c} 
	= \delta^{-1} c + \frac{\partial \chi}{\partial c}(\eta,c).
\end{align}

Moreover, 
according to the principle of energy dissipation, the rate of change of the 
total energy is equal to the dissipation of the system. The following total energy dissipation law is satisfied:$	\frac{\mathrm{d}}{\mathrm{d}t} 	E(\eta,c,\u) = -Q\leq 0,$
indicating that the system is thermodynamically consistent.
For the closed tumor growth system considered here,  the total mass of the system is conserved, that is,
$
	\frac{d}{dt} \int_{\Omega} (\eta+c)\,dx = 0.
$
following \cite{hawkins2012}, we impose the condition $\mathbb{S}_\eta = -\mathbb{S}_c$, which ensures that the 
mass conservation property is satisfied.

By directly computing the time derivative of the energy, we have
\begin{align*}
	\frac{\mathrm{d}E_{fluid}}{\mathrm{d}t}
	&= \frac{1}{2} \int_{\Omega} \bigl( \rho_t |\u|^2 + 2 \rho \u \cdot \u_t \bigr) \,\mathrm{d}x \\
	&= \int_{\Omega} \rho \u \cdot (\u_t + \u \cdot \nabla \u) \,\mathrm{d}x 
	+ \frac{1}{2} \int_{\Omega} \bigl( \rho_t |\u|^2 - 2 \rho \u \cdot (\u \cdot \nabla \u) \bigr) \,\mathrm{d}x \\
	&= \int_{\Omega} (-\nabla \u : \mathbb{T}_\u  -\nabla \u : \sigma_\u) \,\mathrm{d}x 
	+ \frac{1}{2} \int_{\Omega} \bigl( \rho_t |\u|^2 + \nabla \cdot (\rho \u) |\u|^2 \bigr) \,\mathrm{d}x \\
	&= \int_{\Omega} (-\nabla \u : \mathbb{T}_\u -\nabla \u : \sigma_\u) \,\mathrm{d}x 
	- \int_{\Omega} p \textbf{I} : \nabla \u \,\mathrm{d}x.
\end{align*}
where pressure $p$ is introduced as a Lagrange multiplier for incompressibility and we have used the mass conservation equation
$
	\partial_t \rho + \nabla \cdot (\rho \mathbf{u}) = 0.
$
\begin{align*}
	\frac{\mathrm{d}E_{tumor}}{\mathrm{d}t} 
	=& \int_{\Omega} \lambda \eta \frac{\partial \eta}{\partial t} \, \mathrm{d}x 
	-\int_{\Omega} \epsilon^2 \Delta \eta \frac{\partial \eta}{\partial t} \, \mathrm{d}x 
	+ \int_{\Omega} \delta^{-1} c \frac{\partial c}{\partial t} \, \mathrm{d}x 
	+ \int_{\Omega} f'(\eta) \frac{\partial \eta}{\partial t} \mathrm{d}x\\ 
	&+ \int_{\Omega} \partial_{\eta}\chi(\eta,c)\frac{\partial \eta}{\partial t} \, \mathrm{d}x +\int_{\Omega} \partial_{c}\chi(\eta,c)\frac{\partial c}{\partial t} \, \mathrm{d}x  \notag\\
	=&\int_{\Omega} \mu_\eta  \frac{D\eta}{D t} \mathrm{d}x 
	+ \int_{\Omega} \mu_c  \frac{Dc}{D t}\mathrm{d}x 
	- \int_{\Omega} \mu_\eta (\mathbf{u} \cdot \nabla \eta) \mathrm{d}x 
	- \int_{\Omega}\mu_c (\mathbf{u} \cdot \nabla c) \mathrm{d}x \nonumber \\
	= & \int_{\Omega} \nabla \mu_\eta \cdot \mathbb{T}_\eta \, dx+\int_{\Omega} \nabla \mu_c \cdot \mathbb{T}_c \, dx
	+ \int_{\Omega} (\mu_\eta - \mu_c) \cdot S_\eta \mathrm{d}x\\
	&-\int_{\partial \Omega}\big( \frac{1}{2\delta} c^2  + \chi(\eta,c)\big) (\u\cdot \n) dS+ \int_{\Omega} \big( \frac{1}{2\delta} c^2  + \chi(\eta,c)\big) (\nabla \cdot \u) dx\\
	& - \int_{\Omega} \big( f'(\eta) - \epsilon^2 \Delta \eta
	+ \lambda \eta \big)\,
	\mathbf{u} \cdot \nabla \eta \, dx \\
	= & \int_{\Omega} \nabla \mu_\eta \cdot \mathbb{T}_\eta \, dx+\int_{\Omega} \nabla \mu_c \cdot \mathbb{T}_c \, dx
	+ \int_{\Omega} (\mu_\eta - \mu_c) \cdot S_\eta \mathrm{d}x\\
	& \quad
	- \int_{\Omega} \mathbf{u} \cdot \nabla
	\Big( f(\eta) + \tfrac{\lambda}{2} \eta^2 \Big) \, dx  + \int_{\Omega} \epsilon^2 \Delta \eta \nabla \eta\cdot \mathbf{u} \, dx \\
	= & \int_{\Omega} \nabla \mu_\eta \cdot \mathbb{T}_\eta \, dx+\int_{\Omega} \nabla \mu_c \cdot \mathbb{T}_c \, dx
	+ \int_{\Omega} (\mu_\eta - \mu_c) \cdot S_\eta \mathrm{d}x \\
	& \quad
	- \int_{\Omega} \mathbf{u} \cdot \nabla
	\Big( f(\eta) + \tfrac{\lambda}{2} \eta^2 
	+ \tfrac{\epsilon^2}{2} |\nabla \eta|^2 \Big) \, dx  + \int_{\Omega} \epsilon^2
	\nabla \cdot (\nabla \eta \otimes \nabla \eta)
	\cdot \mathbf{u} \, dx \\
	=  & \int_{\Omega} \nabla \mu_\eta \cdot \mathbb{T}_\eta \, dx+\int_{\Omega} \nabla \mu_c \cdot \mathbb{T}_c \, dx
	+\int_{\Omega} (\mu_\eta - \mu_c) \cdot S_\eta \mathrm{d}x 
	- \int_{\Omega} \epsilon^2
	(\nabla \eta \otimes \nabla \eta) : \nabla \mathbf{u} \, dx .
\end{align*}

A combination of these two parts yields
\begin{align}
	\frac{\mathrm{d}}{\mathrm{d}t} 	E(\eta,c,\u) =&\int_{\Omega} (-\nabla \u : \mathbb{T}_\u -\nabla \u : \sigma_\u) \,\mathrm{d}x 
	- \int_{\Omega} p\textbf{I} : \nabla \u \,\mathrm{d}x \\
	&
	+ \int_{\Omega}\nabla \mu_\eta \cdot \mathbb{T}_\eta \mathrm{d}x+ \int_{\Omega} \nabla \mu_c \cdot \mathbb{T}_c \mathrm{d}x 
	+ \int_{\Omega} (\mu_\eta - \mu_c) \cdot S_\eta \mathrm{d}x.\notag\\
	&- \int_{\Omega} \epsilon^2
	(\nabla \eta \otimes \nabla \eta) : \nabla \mathbf{u} \, dx.\notag
\end{align}

So the following is a proper choice of mass fluxes and reaction terms to have the rate
of change of the total energy decaying with time:
\begin{align}
	\mathbb{T}_\u &= 2 \mathbb{D}(\mathbf{u}) -p\textbf{I},\\
	\sigma_\u
	&= -\epsilon^2
	(\nabla \eta \otimes \nabla \eta), \\
	\mathbb{T}_\eta &= -\nabla \mu_\eta, \\
	\mathbb{T}_c &= -\nabla \mu_c, \\
	\mathbb{S}_\eta &= -P(\eta) (\mu_\eta - \mu_c).
\end{align}

In particular, the proliferation 
function $P(\eta)$ is assumed to be nonnegative. 
The function $P$ is defined to be
$$
P(\eta) = 
\begin{cases} 
\delta P_0 \eta, & \text{if } \eta \ge 0, \\ 
0, & \text{otherwise}, 
\end{cases}
$$
where $P_0$ is a positive constant denoting the proliferation parameter.
The total dissipation functional of the system (\ref{tumor-system-1}) is composed of the contributions from 
fluid friction, as well as the dissipative effects arising from phase mixing, diffusion, and reaction processes.
\begin{equation*}
	Q = \int_{\Omega} 2\alpha_\u |\mathbb{D}(\u)|^2 \, \mathrm{d}\mathbf{x} +
	\int_{\Omega} \alpha_\eta \, |\nabla \mu_\eta|^2 \, \mathrm{d}x 
	+ \int_{\Omega} \alpha_c \, |\nabla \mu_c|^2 \, \mathrm{d}x 
	+ \int_{\Omega} P(\eta) \cdot (\mu_\eta-\mu_c)^2 \, \mathrm{d}x,
\end{equation*}
where $\alpha_\u, \alpha_\eta$ and $\alpha_c$ denote the fluid viscosity, phase-field mobility, and nutrient mobility, respectively.
We set $\alpha_\u=\alpha_\eta=\alpha_c=1$ without loss of generality.
For the dissipation induced by the reaction mechanism, we assume that the proliferation 
rate of tumor cells depends on their own volume fraction. 

Noticing that \cite{lin2011,lin2014}
\begin{align}
	&\epsilon^2 \nabla \cdot (\nabla \eta \otimes \nabla \eta)\\
	=&	\epsilon^2 \Delta \eta \, \nabla \eta
	+ \tfrac{	\epsilon^2}{2} \nabla |\nabla \eta|^2 \notag\\
	=&- \mu_\eta \nabla \eta-\mu_c \nabla c + \nabla \big( \frac{\lambda}{2} \eta^2+f(\eta)+\chi (\eta,c)+\frac{\delta^{-1}}{2} c^2+\frac{\epsilon^2}{2}|\nabla \eta|^2 \big)\notag
\end{align}

Especially, we choose $\chi(\eta,c)=-\chi_0 \eta c$, and $\tilde{p}=p+ \frac{\lambda}{2} \eta^2+f(\eta)+\chi (\eta,c)+\frac{\delta^{-1}}{2} c^2+\frac{\epsilon^2}{2}|\nabla \eta|^2$, still denote $\tilde{p}$ by $p$. The following PDE system is formulated:
\begin{align}\label{tumor-model}
	\begin{cases}
		&\eta_t + \mathbf{u} \cdot \nabla \eta - \Delta \mu_\eta + P(\eta) (\mu_\eta - \mu_c) = 0,\\
		&\mu_\eta = \lambda \eta - \epsilon^2 \Delta \eta -\chi_0 c + f'(\eta),\\
		&c_t + \mathbf{u} \cdot \nabla c - \Delta \mu_c -P(\eta) (\mu_\eta - \mu_c) = 0,\\
		&\mu_c = \delta^{-1} c -\chi_0 \eta,\\
		&\mathbf{u}_t + \mathbf{u} \cdot \nabla \mathbf{u} - \Delta \mathbf{u} + \nabla p 
		=\mu_\eta \nabla \eta+\mu_c \nabla c,\\
		&\nabla \cdot \mathbf{u} = 0,\\
		&\mathbf{u} \big|_{\partial \Omega}= 0, \quad\frac{\partial \eta}{\partial \n} \big|_{\partial \Omega}= \frac{\partial c}{\partial \n} \big|_{\partial \Omega}= 0,\quad \frac{\partial \mu_\eta}{\partial \n} \big|_{\partial \Omega}= \frac{\partial \mu_c}{\partial \n} \big|_{\partial \Omega}= 0.
	\end{cases}
\end{align}

Based on $	\frac{\mathrm{d}}{\mathrm{d}t} 	E(\eta,c,\u) = -Q\leq 0$, the model
satisfies the following energy dissipation law:
\begin{align*}
	\frac{d}{dt}E(\eta,c,\u) 
	=&- \big(\int_{\Omega} |\nabla\u|^2 \, \mathrm{d}\mathbf{x} +
	\int_{\Omega}   |\nabla \mu_\eta|^2 \, \mathrm{d}x 
	+ \int_{\Omega}  \, |\nabla \mu_c|^2 \, \mathrm{d}x 
	+ \int_{\Omega} P(\eta) \cdot (\mu_\eta-\mu_c)^2 \, \mathrm{d}x\big)
	\notag\\
	\leq &0.\notag
\end{align*}

\section{Preliminaries and The MSAV Formulation}
For $k\in N^+$ and $1\leq p\leq +\infty$, we denote $L^p(\Omega)$ and $W^{k, p}(\Omega)$ as the classical Lebesgue space and Sobolev space, respectively. The norms of these spaces are denoted by 
within this context, $W^{k, 2}(\Omega)$ is also known as a Hilbert space and can be expressed as $H^k(\Omega)$.  $||\cdot||_{L^\infty}$ represents the norm of the space  $L^\infty(\Omega)$ which is defined as $	||u||_{L^\infty(\Omega)}=ess\sup\limits_{x\in \Omega}|u(x)|.$ For simplicity, we denote the inner products of both
$L^2(\Omega)$ namely, $ (u,v)=\int_\varOmega u(x)v(x) d x,  \forall  \, u,v\in L^2(\Omega)$.
%
%

%

By rewriting the product $ab$ as $
ab = \left((2\epsilon)^{1/2} a \right)
\left(\frac{b}{(2\epsilon)^{1/2}}\right)
$
and applying the elementary inequality
$
xy \le \frac{1}{2}x^2 + \frac{1}{2}y^2,
$
we obtain the following well-known inequalities.

\textbf{(Cauchy’s inequality with $\epsilon$).}
For any $a,b \in \mathbb{R}$ and $\epsilon>0$, it holds that
\begin{align}\label{cauchy-inequality}
	ab \leq \epsilon a^2 + \frac{b^2}{4\epsilon}.
\end{align}

We need the following regularity assumptions for the convergence analysis. 

\textbf{A1:} Assume the exact solutions satisfy the following regularities
\begin{subequations}\label{tumor-regularity}
	\begin{align}
		&\u \in L^\infty (0,T; H^2), \quad \u_t\in L^2(0,T;H^1\cap L^2),\quad \u_{tt}\in L^2(0,T;L^2), \\
		&p\in L^2 (0,T; L^2\cap H^1),\quad p_t\in L^2(0,T;H^1), \\
		&\eta \in L^{\infty}(0,T; H^2), \quad \eta_t\in L^2(0,T;H^1\cap L^2),\quad \eta_{tt}\in L^2(0,T;L^2),\\
		&c\in L^{\infty} (0,T;H^2), \quad c_t\in L^2(0,T;H^1\cap L^2),\quad c_{tt}\in L^2(0,T;L^2),\\
        &r_{tt},q_{tt}\in L^2(0,T).
	\end{align}
\end{subequations}
\textbf{A2:}  Considering that tumors grow only with a certain level of nutrition, we assume that  $P(\eta)$  is nonnegative and satisfies the linear growth condition:
\begin{align}\label{tumor-p-assumption}
	0 \le P(\eta) \le C(1 + |\eta|).
\end{align}

Moreover,  $P(\eta)$ is Lipschitz continuous:
\begin{align}
	|P(\eta_1) - P(\eta_2)| \le C|\eta_1 - \eta_2|.
\end{align}

Let $C_1>0$ be a positive constant, $f(\eta)=\kappa\eta^2(1-\eta)^2-\frac{\lambda}{2}\eta^2$ for a positive constant $\kappa$ and $E_1 (\eta)=\int_{\Omega} f(\eta) dx\geq -C_0$ for some $C_0>0$. Let $C_1>C_0$, so that $E_1(\eta)+C_1>0$ and $E_1+C_1$ has a positive lower bound $C^{'}_0=C_1-C_0$, which we still denote as $C_0$. We introduce the following two scalar auxiliary variables:
$	r(t)=\sqrt{E_1 (\eta)+C_1}, \,  C_1 >C_0, \quad
	q(t)=\exp(-\frac{t}{T}),
$
reformulate the system (\ref{tumor-model}) as 
\begin{align}
		&\eta_t + \frac{r}{\sqrt{E_1 (\eta)+C_1}}\mathbf{u} \cdot \nabla \eta - \Delta \mu_\eta +P(\eta) (\mu_\eta - \mu_c) = 0,\label{tumor-continuous1}\\
		&\mu_\eta = \lambda \eta - \epsilon^2 \Delta \eta -\chi_0 c + \frac{r}{\sqrt{E_1 (\eta)+C_1}}f'(\eta),\label{tumor-continuous2}\\
		&r_t=\frac{1}{2\sqrt{E_1 (\eta)+C_1}} \int_{\Omega} f'(\eta)\eta_tdx,\label{tumor-continuous3}\\
		&c_t +\frac{r}{\sqrt{E_1 (\eta)+C_1}} \mathbf{u} \cdot \nabla c - \Delta \mu_c - P(\eta) (\mu_\eta - \mu_c) = 0,\label{tumor-continuous4}\\
		&\mu_c = \delta^{-1} c -\chi_0 \eta,\label{tumor-continuous5}\\
		&\mathbf{u}_t +\exp(\frac{t}{T})q(t) \mathbf{u} \cdot \nabla \mathbf{u} - \Delta \mathbf{u} + \nabla p 
		\notag\\
		&\quad=\frac{r}{\sqrt{E_1 (\eta)+C_1}}\mu_\eta \nabla \eta+\frac{r}{\sqrt{E_1 (\eta)+C_1}}\mu_c \nabla c,\label{tumor-continuous6}\\
		&\nabla \cdot \mathbf{u} = 0,\label{tumor-continuous7}\\
		&q_t=-\frac{1}{T}q+\exp(\frac{t}{T})\int_{\Omega} (\u\cdot \nabla) \u \cdot \u dx. \label{tumor-continuous8}
\end{align}

\begin{remark}
Since
$
\int_{\Omega} (\mathbf{u}\cdot\nabla)\mathbf{u}\cdot \mathbf{u}\,dx = 0,
$
it is straightforward to verify that, by choosing
$
r_0=\sqrt{E_1(\eta|_{t=0})+C_1}, q(0)=1,
$
the reformulated system is equivalent to the original one.

Taking the $L^2$-inner product of \eqref{tumor-continuous1} with $\mu_\eta$,
\eqref{tumor-continuous2} with $\partial_t \eta$,
\eqref{tumor-continuous3} with $2r$,
\eqref{tumor-continuous4} with $\mu_c$,
\eqref{tumor-continuous5} with $\partial_t c$,
\eqref{tumor-continuous6} with $\mathbf{u}$,
and \eqref{tumor-continuous8} with $q$, respectively, and summing up all the resulting identities,
we obtain the following equivalent energy dissipation law:
\begin{align*}
	\frac{d}{dt}\,\tilde{E}(\eta,c,\mathbf{u},q,r)
	= -\|\nabla \mu_\eta\|_{L^2}^2
	-\|\nabla \mu_c\|_{L^2}^2
	-\|\nabla \mathbf{u}\|_{L^2}^2
	-\frac{1}{T}q^2-\|\sqrt{P(\eta)}(\mu_\eta-\mu_c)\|_{L^2}^2.
\end{align*}
Here, the modified energy functional $\tilde{E}$ is defined as
\begin{align*}
	\tilde{E}(\eta,c,\mathbf{u},q,r)
	= \int_{\Omega} \frac{\rho}{2}|\mathbf{u}|^2\,dx
	+ \int_{\Omega} \left(
	\frac{\lambda}{2}\eta^2
	+ \frac{\epsilon^2}{2}|\nabla \eta|^2
	+ \frac{1}{2\delta}c^2
	-\chi_0 \eta c
	\right)dx
	+ r^2+ \frac{1}{2}q^2.
\end{align*}

	The above energy dissipation law indicates that the reformulated system 
	admits a modified dissipation law as the original model, and the total modified energy
	is non-increasing in time. 
\end{remark}

Let $N$ be a positive integer, and let
$
0 = t_0 < t_1 < \cdots < t_N = T
$
be a uniform partition of the time interval $[0,T]$, with time step size
$\tau = \Delta t = T/N$.
For convenience, we introduce the backward difference operator
$
D_\tau v^{n+1} = \frac{v^{n+1} - v^n}{\tau}.
$
In the following, we shall develop an efficient numerical scheme for the system
\eqref{tumor-continuous1}--\eqref{tumor-continuous8}, which is designed to inherit
the above energy dissipation property at the discrete level.

%
%

\textbf{First-order MSAV scheme:}

Find $(\eta^{n+1},\mu_{\eta}^{n+1},c^{n+1},\mu_{c}^{n+1},\tilde{\u}^{n+1},\u^{n+1},p^{n+1},r^{n+1},q^{n+1})$ such that
\begin{align}
	&\frac{\eta^{n+1}-\eta^n}{\tau} +\frac{r^{n+1}}{\sqrt{E_1 (\eta^n)+C_1}}\mathbf{u}^n \cdot \nabla \eta^n - \Delta \mu_\eta^{n+1} +P(\eta^n) (\mu_\eta^{n+1} - \mu_c^{n+1})=0,\label{tumor-algorithm1}\\
		&\mu_\eta^{n+1} = \lambda \eta^{n+1} - \epsilon^2 \Delta \eta^{n+1} -\chi_0 c^{n+1} + \frac{r^{n+1}}{\sqrt{E_1 (\eta^{n})+C_1}}f'(\eta^{n}),\label{tumor-algorithm2}\\
	&\frac{r^{n+1}-r^n}{\tau}=\frac{1}{2\sqrt{E_1 (\eta^{n})+C_1}} \int_{\Omega} f'(\eta^{n})\frac{\eta^{n+1}-\eta^n}{\tau}dx\label{tumor-algorithm3}\\
	&\quad \quad \quad \quad \quad \quad+\frac{1}{2\sqrt{E_1 (\eta^{n})+C_1}}(\mu_\eta^{n+1},\u^n\cdot\nabla \eta^n)\notag \\
	&\quad \quad \quad \quad\quad \quad+\frac{1}{2\sqrt{E_1 (\eta^{n})+C_1}}(\mu_c^{n+1},\u^n\cdot\nabla c^n)\notag\\
	&\quad \quad \quad \quad \quad \quad
	-\frac{1}{2\sqrt{E_1 (\eta^{n})+C_1}}(\tilde{\u}^{n+1},\mu_\eta^n\nabla \eta^n)\notag\\
	&\quad \quad \quad \quad \quad \quad-\frac{1}{2\sqrt{E_1 (\eta^{n})+C_1}}(\tilde{\u}^{n+1},\mu_c^n\nabla c^n),\notag\\
	&\frac{c^{n+1}-c^n}{\tau} +\frac{r^{n+1}}{\sqrt{E_1 (\eta^n)+C_1}} \mathbf{u}^{n} \cdot \nabla c^{n} - \Delta \mu_c^{n+1} - P(\eta^{n}) (\mu_\eta^{n+1} - \mu_c^{n+1}) = 0,\label{tumor-algorithm4}\\
	&\mu_c^{n+1} = \delta^{-1} c^{n+1} -\chi_0 \eta^{n},\label{tumor-algorithm5}\\
	&\frac{\tilde{\u}^{n+1}-\u^{n}}{\tau} +\exp(\frac{t_{n+1}}{T})q^{n+1} \mathbf{u}^{n} \cdot \nabla \mathbf{u}^{n} - \Delta \mathbf{\tilde{\u}}^{n+1} + \nabla p^{n} \label{tumor-algorithm6}\\
	&\quad
	=\frac{r^{n+1}}{\sqrt{E_1 (\eta^{n})+C_1}}\mu_\eta^{n} \nabla \eta^{n}+\frac{r^{n+1}}{\sqrt{E_1 (\eta^{n})+C_1}}\mu_c^{n} \nabla c^{n},\notag\\
	&\frac{\u^{n+1}-\tilde{\u}^{n+1}}{\tau}+\nabla(p^{n+1}-p^{n})=0,\label{tumor-algorithm8}\\
        &\nabla\cdot \u^{n+1}=0,\quad
\u^{n+1}\cdot\mathbf \n|_{\partial\Omega}=0, \quad \tilde{\u}^{n+1}\big|_{\partial\Omega}=0, \quad
\int_\Omega p^{n+1}\,dx=0,\\
	&\frac{q^{n+1}-q^{n}}{\tau}=-\frac{1}{T}q^{n+1}+\exp(\frac{t_{n+1}}{T})\int_{\Omega} (\u^{n}\cdot \nabla ) \u^{n} \cdot \tilde{\u}^{n+1} dx \label{tumor-algorithm9}.
\end{align}

\begin{remark}
	We note that the term
$
	(\mu_\eta^{n+1},\,\mathbf{u}^n\cdot\nabla \eta^n)
	-
	(\tilde{\mathbf{u}}^{n+1},\,\mu_\eta^n\nabla \eta^n)
	$
	and 
	$
	(\mu_c^{n+1},\,\mathbf{u}^n\cdot\nabla c^n)
	-
	(\tilde{\mathbf{u}}^{n+1},\,\mu_c^n\nabla c^n)
	$
	provides a first-order approximation to
	$
	(\mu_\eta,\,\mathbf{u}\cdot\nabla \eta)
	-
	(\mathbf{u},\,\mu_\eta\nabla \eta)=0$ and $	(\mu_c,\,\mathbf{u}\cdot\nabla c)
	-
	(\mathbf{u},\,\mu_c\nabla c)
	= 0.
	$
	The resulting first-order MSAV scheme is a unconditionally energy-stable numerical method for the coupled system.
	By introducing auxiliary variables 
	$(r,q)$ to reformulate the nonlinear terms, the scheme effectively eliminates the need for expensive nonlinear iterations while strictly preserving thermodynamic consistency.
	Moreover, the velocity and pressure variables are efficiently decoupled through a projection method, which significantly reduces the overall computational cost.
The unique solvability of the discrete scheme \eqref{tumor-algorithm1}--\eqref{tumor-algorithm9} follows from the Lax--Milgram theorem. Since the argument is standard, the details are omitted for brevity.
\end{remark}

\subsection{Algorithm Implementation}
To circumvent the computational difficulties associated with directly solving the fully coupled system
\eqref{tumor-algorithm1}--\eqref{tumor-algorithm9},
we introduce auxiliary variables $\xi^{n+1}_1$ and $\xi^{n+1}_2$ to  decompose the original scheme into a sequence of subproblems,
where 
$
	\xi^{n+1}_1=\frac{r^{n+1}}{\sqrt{E_1 (\eta^n)+C_1}},\, \xi^{n+1}_2=\exp(\frac{t_{n+1}}{T})q^{n+1}.
$
Using $\xi^{n+1}_1, \xi^{n+1}_2$, we split $(\eta^{n+1},\mu_{\eta}^{n+1},c^{n+1},\mu_{c}^{n+1},\tilde{\u}^{n+1},\u^{n+1},p^{n+1})$ into the form: $\gamma^{n+1}=\gamma^{n+1}_0 + \xi^{n+1}_1 \gamma_1^{n+1}+\xi^{n+1}_2 \gamma_2^{n+1}.$
Plugging them into the first-order MSAV algorithm (\ref{tumor-algorithm1})--(\ref{tumor-algorithm9}), the decomposition is implemented through the following steps.

\textbf{Step I:} 
Find $(\eta^{n+1}_0,\mu_{\eta 0}^{n+1},c_0^{n+1},\mu_{c0}^{n+1})$ by:
\begin{align}
	\begin{cases}
		\frac{\eta^{n+1}_0-\eta^n}{\tau}-\Delta \mu_{\eta0}^{n+1}+P(\eta^n)(\mu^{n+1}_{\eta0}-\mu^{n+1}_{c0})=0,\\
		\mu_{\eta0}^{n+1}=\lambda \eta^{n+1}_0-\epsilon^2\Delta\eta^{n+1}_0-\chi_0  c_0^{n+1},\\
		\frac{c^{n+1}_0-c^n}{\tau}-\Delta\mu_{c0}^{n+1}-P(\eta^n)(\mu^{n+1}_{\eta0}-\mu^{n+1}_{c0})=0,\\
		\mu^{n+1}_{c0}=\delta^{-1}c^{n+1}_0-\chi_0 \eta^n.
			\end{cases}
\end{align}

Find $(\eta^{n+1}_1,\mu_{\eta 1}^{n+1},c_1^{n+1},\mu_{c1}^{n+1})$ by:
\begin{equation}\label{step2-system}
	\begin{cases}
		\dfrac{\eta_1^{n+1}}{\tau}
		+ \u^n \cdot \nabla \eta^n
		- \Delta \mu_{\eta1}^{n+1}
		+ P(\eta^n)\bigl(\mu_{\eta1}^{n+1} - \mu_{c1}^{n+1}\bigr)
		= 0, \\
		\mu_{\eta1}^{n+1}
		= \lambda \eta_1^{n+1}
		- \epsilon^2 \Delta \eta_1^{n+1}
		- \chi_0 c_1^{n+1}
		+ f'(\eta^n), \\		
		\dfrac{c_1^{n+1}}{\tau}
		+ \u^n \cdot \nabla c^{n}
		- \Delta \mu_{c1}^{n+1}
		- P(\eta^n)\bigl(\mu_{\eta1}^{n+1} - \mu_{c1}^{n+1}\bigr)
		= 0, \\
		\mu_{c1}^{n+1}
		= \delta^{-1} c_1^{n+1}.
	\end{cases}
\end{equation}

Find $(\eta^{n+1}_2,\mu_{\eta 2}^{n+1},c_2^{n+1},\mu_{c2}^{n+1})$ by:
\begin{align}
	\begin{cases}
		\frac{\eta_2^{n+1}}{\tau} - \Delta \mu_{\eta2}^{n+1} + P(\eta^{n}) (\mu_{\eta2}^{n+1} - \mu_{c2}^{n+1}) = 0,\\
		\mu_{\eta2}^{n+1} = \lambda \eta_2^{n+1} - \epsilon^2 \Delta \eta_2^{n+1} - \chi_0 c_2^{n+1},\\
		\frac{c_2^{n+1}}{\tau} - \Delta \mu_{c2}^{n+1} - P(\eta^{n}) (\mu_{\eta2}^{n+1} - \mu_{c2}^{n+1}) = 0,\\
		\mu_{c2}^{n+1} = \delta^{-1} c_2^{n+1}.
	\end{cases}
\end{align}

\textbf{Step II:} 
Find $(\tilde{\u}_0^{n+1},\u_0^{n+1},p_0^{n+1})$ by:
\begin{align}
	\begin{cases}
		\frac{\tilde{\u}^{n+1}_0-\u^n}{\tau}-\Delta \tilde{\u}^{n+1}_0+\nabla p^n=0,\\
		\frac{\u^{n+1}_0-\tilde{\u}^{n+1}_0}{\tau} +\nabla (p^{n+1}_0-p^n)=0.
	\end{cases}
\end{align}

Find $(\tilde{\u}_1^{n+1},\u_1^{n+1},p_1^{n+1})$ by:
\begin{align}
	\begin{cases}
		\dfrac{\tilde{\u}_1^{n+1}}{\tau}
		- \Delta \tilde{\u}_1^{n+1}
		= \mu_\eta^{n} \nabla \eta^{n}+\mu_c^n \nabla c^n, \\
		\frac{\u^{n+1}_1-\tilde{\u}^{n+1}_1}{\tau}
		+ \nabla p_1^{n+1}
		= 0.
	\end{cases}
\end{align}

Find $(\tilde{\u}_2^{n+1},\u_2^{n+1},p_2^{n+1})$ by:
\begin{align}
\begin{cases}
\frac{\tilde{\u}_2^{n+1} }{\tau} + \u^n \cdot \nabla \u^n - \Delta \tilde{\u}_2^{n+1} = 0,\\
\frac{\u_2^{n+1} - \tilde{\u}_2^{n+1}}{\tau} + \nabla p_2^{n+1} = 0.
\end{cases}
\end{align}

\textbf{Step III:} After getting $(\eta^{n+1}_i,\mu_{\eta i}^{n+1},c_i^{n+1},\mu_{ci}^{n+1},\tilde{\u}_i^{n+1},\u_i^{n+1},p_i^{n+1})$, it is essential to determine the auxiliary variables $(\xi^{n+1}_1,\xi^{n+1}_2)$. From (\ref{tumor-algorithm3}) and (\ref{tumor-algorithm9}), we find that  $(\xi^{n+1}_1,\xi^{n+1}_2)$ can be obtained by solving a $2\times2$ linear algebraic system.
\begin{align}
	\begin{cases}
			\alpha_1 \xi^{n+1}_1+\alpha_2 \xi^{n+1}_2=\alpha_0,\\
		\beta_1 \xi^{n+1}_1+\beta_2 \xi^{n+1}_2=\beta_0.
	\end{cases}
\end{align}
where
\begin{align*}
	\begin{cases}
	\alpha_1&=\frac{\sqrt{E_1 (\eta^n) + C_1}}{\tau} - \frac{1}{2\sqrt{E_1 (\eta^n) + C_1}} ( (f'(\eta^n),\frac{\eta_1^{n+1}}{\tau})+ (\mu_{\eta1}^{n+1},\u^n\cdot \nabla \eta^n)- (\tilde{\u}^{n+1}_1,\mu_\eta^n\nabla \eta^n) \notag\\
	&\quad+(\mu_{c1}^{n+1},\u^n\cdot \nabla c^n)- (\tilde{\u}^{n+1}_1,\mu_c^n \nabla c^n)),\\
	\alpha_2&= -\frac{1}{2\sqrt{E_1 (\eta^n) + C_1}} ((f'(\eta^n),\frac{\eta_2^{n+1}}{\tau})+ (\mu_{\eta2}^{n+1},\u^n\cdot \nabla \eta^n)- (\tilde{\u}^{n+1}_2,\mu_\eta^n \nabla \eta^n)\notag\\
	&\quad+(\mu_{c2}^{n+1},\u^n\cdot \nabla c^n)- (\tilde{\u}^{n+1}_2,\mu_c^n \nabla c^n)),\\
	\alpha_0&=  \frac{r^n}{\tau} + \frac{1}{2\sqrt{E_1 (\eta^n) + C_1}} ((f'(\eta^n),\frac{\eta_0^{n+1}-\eta^n}{\tau})+ (\mu_{\eta0}^{n+1},\u^n\cdot \nabla \eta^n)- (\tilde{\u}^{n+1}_0,\mu_\eta^n \nabla \eta^n)\notag\\
	&\quad+(\mu_{c0}^{n+1},\u^n\cdot \nabla c^n)- (\tilde{\u}^{n+1}_0,\mu_c^n \nabla c^n)),\\
	\beta_1&=-\exp\left(\frac{t_{n+1}}{T}\right) \left(\u^n \cdot \nabla \u^n, \tilde{\u}_1^{n+1}\right),\\
	\beta_2&=\left(\frac{1}{\tau}+\frac{1}{T}\right) \exp\left(-\frac{t_{n+1}}{T}\right) - \exp\left(\frac{t_{n+1}}{T}\right) \left(\u^n \cdot \nabla \u^n, \tilde{\u}_2^{n+1}\right),\\
	\beta_0&=\frac{q^n}{\tau} + \exp\left(\frac{t_{n+1}}{T}\right) \left(\u^n \cdot \nabla \u^n, \tilde{\u}^{n+1}_0\right).
		\end{cases}
\end{align*}

\subsection{Mass conservation and energy stability}
The mass-conservation and energy stability of the first-order MSAV algorithm (\ref{tumor-algorithm1})--(\ref{tumor-algorithm9})  is established in the following theorem.
\begin{theorem}
The scheme (\ref{tumor-algorithm1})–(\ref{tumor-algorithm9}) satisfies the mass conservation in the sense that
 	\begin{align}\label{tumor-mass}
 	\int_{\Omega} (\eta^{n+1}+c^{n+1}) dx =\int_{\Omega} (\eta^0+c^0) dx.
 \end{align}
 
Moreover, it is unconditionally energy stable in the sense that
\begin{align}\label{tumor-energy-inequality}
	E^{n+1}\leq E^{n},
\end{align}
where $E^{n+1}=\frac{\lambda}{2} \| \eta^{n+1}\|^2_{L^2}+\frac{\epsilon^2}{2} \| \nabla \eta^{n+1}\|^2_{L^2}+|r^{n+1}|^2+\frac{\delta^{-1}}{2}\|c^{n+1}\|^2_{L^2}+\frac{1}{2}\|\u^{n+1}\|^2_{L^2}+\frac{\tau^2}{2}\|\nabla p^{n+1}\|^2_{L^2}+\frac{1}{2}|q^{n+1}|^2-\chi_0(\eta^{n+1}, c^{n+1}).$
\end{theorem}

\begin{proof}
For the first-order scheme, integrating \eqref{tumor-algorithm1} and \eqref{tumor-algorithm4} over the domain $\Omega$ 
and summing the resulting identities yield \eqref{tumor-mass}.
Taking the inner product of (\ref{tumor-algorithm1}) and (\ref{tumor-algorithm2}) with $\tau \mu_\eta^{n+1}$ and $(\eta^{n+1}-\eta^n)$, we have
\begin{align}\label{tumor-7}
	&\frac{\lambda}{2}(\|\eta^{n+1}\|^2_{L^2}-\|\eta^{n}\|^2_{L^2}+\|\eta^{n+1}-\eta^{n}\|^2_{L^2})-\chi_0  (c^{n+1},\eta^{n+1}-\eta^n)\\
	&\quad+ \frac{\epsilon^2}{2} (\|\nabla \eta^{n+1}\|^2_{L^2}-\|\nabla\eta^{n}\|^2_{L^2}+\|\nabla(\eta^{n+1}-\eta^{n})\|^2_{L^2})\notag\\
	&\quad+\frac{r^{n+1}}{\sqrt{E_1 (\eta^n)+C_1}}( f'(\eta^n),\eta^{n+1}-\eta^n)+ \frac{\tau r^{n+1}}{\sqrt{E_1 (\eta^n)+C_1}}(\u^n\cdot \nabla \eta^n,\mu_{\eta}^{n+1})\notag\\
	&\quad+\tau \| \nabla \mu_{\eta}^{n+1}\|^2_{L^2}+\tau \big( P(\eta^n)(\mu_{\eta}^{n+1}-\mu_c^{n+1}), \mu_{\eta}^{n+1}\big)
	=0.\notag
\end{align}

Testing (\ref{tumor-algorithm3}) by $2\tau r^{n+1}$, we get
\begin{align}\label{tumor-8}
	&( |r^{n+1}|^2-|r^{n}|^2+|r^{n+1}-r^n|^2)\\
	=&\frac{ r^{n+1}}{\sqrt{E_1(\eta^n) + C_1}} (f'(\eta^n), \eta^{n+1} - \eta^n) + \frac{\tau r^{n+1}}{\sqrt{E_1(\eta^n) + C_1}} (\mu_\eta^{n+1}, \u^n \cdot \nabla \eta^n)\notag\\
	&\quad + \frac{\tau r^{n+1}}{\sqrt{E_1(\eta^n) + C_1}} (\mu_c^{n+1}, \u^n \cdot \nabla c^n) - \frac{ \tau r^{n+1}}{\sqrt{E_1(\eta^n) + C_1}} (\tilde{\u}^{n+1}, \mu_\eta^n \nabla \eta^n)\notag\\
	&\quad- \frac{ \tau r^{n+1}}{\sqrt{E_1(\eta^n) + C_1}} (\tilde{\u}^{n+1}, \mu_c^n \nabla c^n).\notag
\end{align}

Taking the inner product of (\ref{tumor-algorithm4}) and (\ref{tumor-algorithm5}) with $\tau \mu_{c}^{n+1}$ and $(c^{n+1}-c^n)$ one has 
\begin{align}\label{tumor-9}
	\frac{\delta^{-1}}{2}(\|c^{n+1}\|_{L^2}^2 - \|c^n\|_{L^2}^2 + \|c^{n+1} - c^n\|_{L^2}^2) - \chi_0 (\eta^n,c^{n+1} -c^n) \\
	+ \frac{\tau r^{n+1}}{\sqrt{E_1 (\eta^n)+C_1}}(\u^n\cdot \nabla c^n,\mu_{c}^{n+1}) + \tau \|\nabla \mu_c^{n+1}\|_{L^2}^2 - \tau\big( P(\eta^n)(\mu_{\eta}^{n+1} - \mu_c^{n+1}), \mu_c^{n+1} \big)= 0.\notag
\end{align}

Testing (\ref{tumor-algorithm6}) by $\tau \tilde{\u}^{n+1}$, we derive
\begin{align}
\frac{1}{2}(	\|\tilde{\u}^{n+1}\|_{L^2}^2 - \|
	\u^n\|_{L^2}^2 + \|\tilde{\u}^{n+1} - \u^n\|_{L^2}^2)+ \tau \exp(\frac{t_{n+1}}{T}) q^{n+1} (\u^n \cdot \nabla \u^n, \tilde{\u}^{n+1}) \notag\\
	+ \tau \|\nabla \tilde{\u}^{n+1}\|_{L^2}^2 + \tau (\nabla p^n, \tilde{\u}^{n+1}) =  \frac{\tau r^{n+1}}{\sqrt{E_1(\eta^n) + C_1}} (\mu_{\eta}^n  \nabla \eta^n, \tilde{\u}^{n+1})\label{tumor-10}\\
	+\quad  \frac{\tau r^{n+1}}{\sqrt{E_1(\eta^n) + C_1}} (\mu_{c}^n  \nabla c^n, \tilde{\u}^{n+1}).\notag
\end{align}

From (\ref{tumor-algorithm8}), we can get
$
		\frac{1}{2}\u^{n+1} +\frac{\tau}{2} \nabla p^{n+1} =\frac{1}{2} \tilde{\u}^{n+1} +\frac{\tau}{2}  \nabla p^n.
$
Taking the inner product of the identity with itself, and noticing that $(\nabla p^{n+1},\u^{n+1})=-(p^{n+1},\nabla \cdot \u^{n+1})=0$, we can deduce 
\begin{align}\label{tumor-11}
\frac{1}{2}(\|\u^{n+1}\|_{L^2}^2 - \|\tilde{\u}^{n+1}\|_{L^2}^2) + \frac{\tau^2}{2} (\|\nabla p^{n+1}\|_{L^2}^2 - \|\nabla p^n\|_{L^2}^2) = \tau (\nabla p^n, \tilde{\u}^{n+1})
\end{align}

Testing (\ref{tumor-algorithm9}) by $\tau q^{n+1}$, we have
\begin{align}
	\frac{1}{2}(|q^{n+1}|^2 - |q^n|^2 + |q^{n+1} - q^n|^2 )+ \frac{\tau}{T} |q^{n+1}|^2\label{tumor-12}\\
	 = \tau \exp\left(\frac{t_{n+1}}{T}\right) q^{n+1} (\u^n \cdot \nabla \u^n, \tilde{\u}^{n+1}).\notag
\end{align}

Summing up (\ref{tumor-7})--(\ref{tumor-12}), we get (\ref{tumor-energy-inequality}).
%
\end{proof}

\section{Error Analysis}
In this section, we perform a rigorous error analysis of the first-order time-discrete scheme \eqref{tumor-algorithm1}--\eqref{tumor-algorithm9}. 
Let $(\eta, \mu_\eta, c, \mu_c, r, \mathbf{u}, p, q)$ denote the exact solution of the continuous system \eqref{tumor-continuous1}--\eqref{tumor-continuous8}, and let $(\eta^{n+1}, \allowbreak\mu_{\eta}^{n+1}, c^{n+1},\allowbreak \mu_{c}^{n+1},\allowbreak \tilde{\mathbf{u}}^{n+1},\allowbreak \mathbf{u}^{n+1}, \allowbreak p^{n+1}, \allowbreak r^{n+1}, q^{n+1})$ be the numerical solution obtained from the discrete scheme \eqref{tumor-algorithm1}--\eqref{tumor-algorithm9}. We define the error terms as $e^{n+1}_{\zeta} = \zeta ^{n+1}-\zeta(t_{n+1})$, where $\zeta$ denotes the primary unknown variable.
We observe that the energy inequalities are insufficient to provide a uniform bound for the time-discrete solutions, due to the presence of the coupling term $2 \chi_0 \eta^{n+1} c^{n+1}$. Consequently, the standard energy argument alone cannot be directly applied. 
To facilitate the convergence analysis, we therefore begin by establishing the following regularity estimates for the time-discrete solutions.
\begin{theorem}
The first-order MSAV schemes (\ref{tumor-algorithm1})–(\ref{tumor-algorithm9}) satisfy the following boundedness results.
\begin{align}\label{tumor-regularity-L2}
		&\frac{\lambda}{2} \| \eta^{n+1}\|^2_{L^2}+\epsilon^2 \| \nabla \eta^{n+1}\|^2_{L^2}+2|r^{n+1}|^2+\frac{\delta^{-1}}{2}\|c^{n+1}\|^2_{L^2}\\
		&+\|\u^{n+1}\|^2_{L^2}+|q^{n+1}|^2+\tau^2\|\nabla p^{n+1}\|^2_{L^2}
		\leq C.\notag
\end{align}
\end{theorem}
\begin{proof}
From (\ref{tumor-energy-inequality}), it follows that
$
E^{n+1} \le E^{0},
$
i.e., the discrete energy is bounded by the initial energy.
\begin{align}
	\begin{aligned}
&\lambda \| \eta^{n+1}\|^2_{L^2}+\epsilon^2 \| \nabla \eta^{n+1}\|^2_{L^2}+2|r^{n+1}|^2+\delta^{-1}\|c^{n+1}\|^2_{L^2}\\
&+\|\u^{n+1}\|_{L^2}+\tau^2\|\nabla p^{n+1}\|^2_{L^2}+|q^{n+1}|^2\\
\leq& 
\lambda \| \eta^{0}\|^2_{L^2}+\epsilon^2 \| \nabla \eta^{0}\|^2_{L^2}+2|r^{0}|^2+\delta^{-1}\|c^{0}\|^2_{L^2}\\
&+\|\u^{0}\|_{L^2}+\tau^2\|\nabla p^{0}\|^2_{L^2}+|q^{0}|^2-2\chi_0(\eta^{0}, c^{0})+2\chi_0(\eta^{n+1}, c^{n+1})
\end{aligned}
\end{align}

Using (\ref{cauchy-inequality}) and noting that $\lambda = 4\delta \chi_0^2$, we obtain
\begin{align}
	2\chi_0(\eta^{n+1},c^{n+1})
	\le2 \delta \chi_0^2 \|\eta^{n+1}\|_{L^2}^2
	+ \frac{\delta^{-1}}{2} \|c^{n+1}\|_{L^2}^2.
\end{align}
Thus, we arrive at (\ref{tumor-regularity-L2}).
\end{proof}

Furthermore, since an $L^2$-bound for $\nabla \mathbf{u}^{n+1}$ cannot be obtained directly at this stage of the analysis and regularity assumptions on $\eta$, $c$ are essential for the subsequent error analysis. we adopt the induction hypothesis
\begin{align}\label{tumor-eta-h2-induction}
\|\nabla\mathbf u^i\|_{L^2}
+\|\Delta\eta^i\|_{L^2}
+\|\Delta c^i\|_{L^2} \leq C,\quad 0 \leq  i \leq n.
\end{align}
thus, from (\ref{tumor-p-assumption}) we can deduce that
\begin{align}\label{tumor-p-eta-infty}
\|P(\eta^n)\|_{\infty}\leq C(1+\|\eta^n\|_{\infty})\leq C.
\end{align}

The corresponding estimates for $\|\nabla \mathbf{u}^{n+1}\|_{L^2}$, $\|\Delta \eta^{n+1}\|_{L^2}$ and $\|\Delta c^{n+1}\|_{L^2}$ will be established by mathematical induction; see \eqref{tumor-u-h1} and Theorem~\ref{tumor-eta-h2-theorem}. 
The main results in this paper are stated in the following theorem:
\begin{theorem}
\label{tumor-main-theorem}
	Under the regularity assumptions \eqref{tumor-regularity}, let $(\eta, \mu_\eta, c, \mu_c,\mathbf{u}, p, r,q)$ denote the exact solution of the continuous system \eqref{tumor-continuous1}--\eqref{tumor-continuous8}, and let $(\eta^{n+1}, \allowbreak\mu_{\eta}^{n+1}, c^{n+1},\allowbreak \mu_{c}^{n+1},\allowbreak \mathbf{u}^{n+1}, \allowbreak p^{n+1}, \allowbreak r^{n+1}, q^{n+1})$ be the numerical solution obtained from the discrete scheme \eqref{tumor-algorithm1}--\eqref{tumor-algorithm9}. Then, the following error estimate holds:
\begin{align}
		& \|e_{\eta}^{N}\|_{L^2}^2 + \|\nabla e_{\eta}^{N}\|_{L^2}^2 +|e^{N}_r|^2
		+ \| e_c^{N} \|_{L^2}^2 +\| e^{N}_\u\|^2_{L^2}
		+|e_{q}^{N}|^{2}  +\tau^2 \|\nabla e_{p}^{N}\|_{L^2}^{2} \notag\\
		&+\tau\sum_{n=0}^{N-1} \big( \delta^{-1} \| \nabla e_c^{n+1} \|_{L^2}^2 + \| e_{\mu c}^{n+1} \|_{L^2}^2+  \| \nabla e_{\mu c}^{n+1} \|_{L^2}^2+ \|e_{\mu \eta}^{n+1}\|_{L^2}^2  \label{tumor-main-results}\\
		&+ \|\nabla e_{\mu \eta}^{n+1}\|_{L^2}^2+\| \nabla \tilde{e}^{n+1}_\u\|^2_{L^2}\big)
		\leq C  \tau^2.\notag
\end{align}
	for all $0 \leq n \leq N-1$, where $C$ is a constant independent of $\tau$.
\end{theorem}

We establish the proof of this main theorem through the following sequence of lemmas:
\begin{lemma}
Under the assumption (\ref{tumor-regularity}), we have
\begin{align}
		&\frac{\lambda + \epsilon^2}{2\tau} \left( \|e_{\eta}^{n+1}\|_{L^2}^2 - \|e_{\eta}^n\|_{L^2}^2 + \|e_{\eta}^{n+1} - e_{\eta}^n\|_{L^2}^2 \right)+ \frac{\epsilon^2}{2\tau} ( \|\nabla e_{\eta}^{n+1}\|_{L^2}^2 - \|\nabla e_{\eta}^n\|_{L^2}^2  \label{tumor-eta-results}\\
		&+ \|\nabla(e_{\eta}^{n+1} - e_{\eta}^n)\|_{L^2}^2 )+ \|e_{\mu \eta}^{n+1}\|_{L^2}^2 + \|\nabla e_{\mu \eta}^{n+1}\|_{L^2}^2+ \|\sqrt{P(\eta^n)}e^{n+1}_{\mu \eta} \|^2_{L^2} \notag\\
		\leq & - \frac{e_r^{n+1}}{\sqrt{E_1(\eta^n) + C_1}} \big(f'(\eta^n), \frac{e_{\eta}^{n+1} - e_{\eta}^n}{\tau}\big)
	+ \epsilon_{\mu \eta} \|\nabla e_{u\eta}^{n+1}\|_{L^2}^2+\big(P(\eta^n) e^{n+1}_{\mu \eta}, e^{n+1}_{\mu c}\big)\notag
	\\
	& + \epsilon_{\mu c} \|\nabla e_{uc}^{n+1}\|_{L^2}^2+C \|\nabla e_c^{n+1} \|_{L^2}^2 + C \tau^2 + C (|e_r^{n+1}|^2 +\|e_c^{n+1}\|_{L^2}^2\notag\\
		&+ \epsilon^2 \|\nabla e_{\eta}^{n+1}\|_{L^2}^2+ \|e_{\eta}^n\|_{L^2}^2+ \|e_\u^n\|_{L^2}^2 + \|\nabla e_{\eta}^n\|_{L^2}^2 +\|R^{n+1}_\eta\|_{L^2}^2).\notag
\end{align}
\end{lemma}
\begin{proof}
Equation (\ref{tumor-continuous1})--(\ref{tumor-continuous2}) at $t=t_{n+1}$ satisfy
\begin{align}
	&\frac{\eta(t_{n+1})-\eta(t_n)}{\tau}+\frac{r(t_{n+1})}{\sqrt{E_1(\eta(t_{n+1})) + C_1}} \u(t_{n+1}) \cdot \nabla \eta(t_{n+1})\label{tumor-continuous1-1}\\
	&\qquad-\Delta \mu_\eta(t_{n+1}) + P(\eta(t_n)) (\mu_\eta(t_{n+1}) - \mu_c(t_{n+1})) = R_\eta^{n+1},\notag\\
	&\mu_\eta(t_{n+1}) = \lambda \eta(t_{n+1}) - \epsilon^2 \Delta \eta(t_{n+1}) - \chi_0 c(t_{n+1})\label{tumor-continuous1-2}\\
	& \qquad \qquad \quad+ \frac{r(t_{n+1})}{\sqrt{E_1(\eta(t_{n+1})) + C_1}} f'(\eta(t_{n+1})). \notag
\end{align}
where
\begin{align*}
	R_\eta^{n+1} &=  \frac{\eta(t_{n+1}) - \eta(t_n)}{\tau}-\frac{\partial \eta}{\partial t}(t_{n+1}) - P(\eta(t_{n+1})) (\mu_\eta(t_{n+1}) - \mu_c(t_{n+1})) \\
	&\quad+ P(\eta(t_n)) (\mu_\eta(t_{n+1}) - \mu_c(t_{n+1})).\notag
\end{align*}

By the assumption (\ref{tumor-regularity}),we have 
\begin{align*}
	\| R_\eta^{n+1} \|_{L^2} &\leq\| \frac{1}{\tau} \int_{t_n}^{t_{n+1}} (t_n - t) \frac{\partial^2 \eta}{\partial t^2} dt \|_{L^2}+ C \| \eta(t_{n+1}) - \eta(t_n) \|_{L^4} (\| \mu_\eta \|_{L^4} + \| \mu_c \|_{L^4})\notag \\
	&\leq C \tau^{\frac{1}{2}}.\notag
\end{align*}

Subtracting (\ref{tumor-continuous1-1})--(\ref{tumor-continuous1-2}) from (\ref{tumor-algorithm1})--(\ref{tumor-algorithm2}), we have the following error equations
\begin{align}
		&\frac{e_\eta^{n+1}-e_\eta^n}{\tau} - \Delta e_{\mu \eta}^{n+1}\label{tumor-15}\\ = &-R^{n+1}_{\eta} - P(\eta^n)(e_{\mu \eta}^{n+1} - e_{\mu c}^{n+1})- \big(P(\eta^n) - P(\eta(t_n))\big)(\mu_\eta(t_{n+1}) - \mu_c(t_{n+1}))\notag\\
		& +\frac{r(t_{n+1})}{\sqrt{E_1(\eta(t_{n+1}))+C_1}} \nabla \cdot \big(\u(t_{n+1})\eta(t_{n+1}) - \u(t_n)\eta(t_n)\big)\notag \\
		& + \big(\frac{r(t_{n+1})}{\sqrt{E_1(\eta(t_{n+1}))+C_1}} - \frac{r^{n+1}}{\sqrt{E_1(\eta^n)+C_1}}\big) \nabla \cdot \big(\u(t_n)\eta(t_n)\big)  \notag\\
		&- \frac{r^{n+1}}{\sqrt{E_1(\eta^n)+C_1}} \nabla \cdot \big(\u(t_n)e_\eta^n\big)- \frac{r^{n+1}}{\sqrt{E_1(\eta^n)+C_1}} \nabla \cdot \big(e_\u^n \eta^n\big)\notag\\
			:= & F_\eta\notag
		\end{align}
and		
		\begin{align}
			e_{\mu \eta}^{n+1} = & \lambda e_\eta^{n+1} - \epsilon^2 \Delta e_\eta^{n+1} - \chi_0 e_c^{n+1} + e_r^{n+1} \frac{f'(\eta^n)}{\sqrt{E_1(\eta^n)+C_1}} \label{tumor-16}\\
			& + \frac{r(t_{n+1}) f'(\eta(t_{n+1})) \big(E_1(\eta(t_{n+1})) - E_1(\eta^n)\big)}{\sqrt{E_1(\eta^n)+C_1} \cdot \sqrt{E_1(\eta(t_{n+1}))+C_1}\big(\sqrt{E_1(\eta^n)+C_1} +\sqrt{E_1(\eta(t_{n+1}))+C_1}\big)} \notag\\
			&+ \frac{r(t_{n+1})}{\sqrt{E_1(\eta^n)+C_1}} \big(f'(\eta^n) - f'(\eta(t_{n+1}))\big)\notag\\
			:=& \lambda e_{\eta}^{n+1} - \epsilon^2 \Delta e_{\eta}^{n+1} - \chi_0 e_{c}^{n+1} + e_{r}^{n+1} \frac{f'(\eta^n)}{\sqrt{E_1(\eta^n) + C_1}} + F_{\mu \eta}.\notag
					\end{align}
Taking the inner product of \eqref{tumor-15} with $e_{\mu \eta}^{n+1}$ and $\epsilon^2 e_{\eta}^{n+1}$, \eqref{tumor-16} with $e_{\mu \eta}^{n+1}$ and $\frac{e_{\eta}^{n+1} - e_{\eta}^n}{\tau}$, respectively.
Summing up the resulting error equations, we have
\begin{align}
		&\frac{\lambda + \epsilon^2}{2\tau} \left( \|e_{\eta}^{n+1}\|_{L^2}^2 - \|e_{\eta}^n\|_{L^2}^2 + \|e_{\eta}^{n+1} - e_{\eta}^n\|_{L^2}^2 \right)+ \|e_{\mu \eta}^{n+1}\|_{L^2}^2  \label{tumor-19}\\
		&+ \frac{\epsilon^2}{2\tau} \left( \|\nabla e_{\eta}^{n+1}\|_{L^2}^2 - \|\nabla e_{\eta}^n\|_{L^2}^2 + \|\nabla(e_{\eta}^{n+1} - e_{\eta}^n)\|_{L^2}^2 \right)+ \|\nabla e_{\mu \eta}^{n+1}\|_{L^2}^2 \notag\\
		= & (F_\eta, e_{\mu\eta}^{n+1} + \epsilon^2 e_{\eta}^{n+1}) +( \lambda e_{\eta}^{n+1}- \chi_0 e_c^{n+1}, e_{\mu \eta}^{n+1})\notag\\
		&+ \frac{f'(\eta^n)}{\sqrt{E_1(\eta^n) + C_1}} (e_r^{n+1}, e_{\mu \eta}^{n+1})
		- \chi_0 \big( e_c^{n+1}, \frac{e_{\eta}^{n+1} - e_{\eta}^n}{\tau} \big) \notag\\
		&- \frac{e_r^{n+1}}{\sqrt{E_1(\eta^n) + C_1}} \big(f'(\eta^n), \frac{e_{\eta}^{n+1} - e_{\eta}^n}{\tau}\big)+ \big( F_{\mu \eta}, e_{\mu \eta}^{n+1}-\frac{e_{\eta}^{n+1} - e_{\eta}^n}{\tau} \big) .\notag
\end{align}

Next, we estimate the above terms one by one. 	Observing that $|r^{n+1}| \leq C$ and $E_1+C_1 >C_0$,
by (\ref{tumor-regularity}), (\ref{tumor-p-eta-infty}), integrate by parts (IBP), together with the H\"{o}lder inequality, and Young's inequality, we deduce that
\begin{align}
		&(F_\eta, e_{\mu\eta}^{n+1} + \epsilon^2 e_{\eta}^{n+1}) \label{tumor-A1}\\ 
		\le &\|e_{\mu\eta}^{n+1} + \epsilon^2 e_{\eta}^{n+1}\|_{L^2}\|R^{n+1}_\eta\|_{L^2} +\big(e_{\mu\eta}^{n+1},- P(\eta^n)(e_{\mu \eta}^{n+1} - e_{\mu c}^{n+1}) \big)\notag \\
		&+C\| \epsilon^2 e_{\eta}^{n+1}\|_{L^2} \| P(\eta^n)\|_{L^{\infty}} \|(e_{\mu \eta}^{n+1} - e_{\mu c}^{n+1})\|_{L^2}\notag\\
		 &+C\|e_{\mu\eta}^{n+1} + \epsilon^2 e_{\eta}^{n+1}\|_{L^3} \|e_{\eta}^n\|_{L^2} \|\mu_{\eta}(t_{n+1}) - \mu_{c}(t_{n+1})\|_{L^6} \notag\\
		& + \|\nabla(e_{\mu\eta}^{n+1} + \epsilon^2 e_{\eta}^{n+1})\|_{L^2} \big(\|\u(t_{n+1})\|_{L^\infty} \|\eta_{t}\|_{L^2} \tau + \| \eta(t_n)\|_{L^\infty} \|\u_{t}\|_{L^2} \tau\big)\notag \\
		& + C \|e_{\mu\eta}^{n+1} + \epsilon^2 e_{\eta}^{n+1}\|_{L^2} |e_r^{n+1}| + C \|\nabla (e_{\mu\eta}^{n+1} + \epsilon^2 e_{\eta}^{n+1})\|_{L^2} (\|e_{\eta}^n\|_{L^2} + \|\eta_{t}\|_{L^2} \tau) \notag\\
		& + C \| e_{\mu\eta}^{n+1} + \epsilon^2 e_{\eta}^{n+1}\|_{L^2} \|\u(t_n)\|_{\infty}\|\nabla e_{\eta}^n\|_{L^2} + C \|\nabla(e_{\mu\eta}^{n+1} + \epsilon^2 e_{\eta}^{n+1})\|_{L^2} \|e_\u^n\|_{L^2} \notag\\
		\le & - \|\sqrt{P(\eta^n)} e^{n+1}_{\mu \eta} \|^2_{L^2}+\big(P(\eta^n) e^{n+1}_{\mu \eta}, e^{n+1}_{\mu c}\big)+ \epsilon_{\mu \eta} \|\nabla e_{u\eta}^{n+1}\|_{L^2}^2 + \epsilon_{\mu c} \|\nabla e_{uc}^{n+1}\|_{L^2}^2\notag \\
		& + C (\tau^2+|e_r^{n+1}|^2 + \|e_{\eta}^n\|_{L^2}^2 + \|e_\u^n\|_{L^2}^2 + \|\nabla e_{\eta}^n\|_{L^2}^2 + \epsilon^2 \|\nabla e_{\eta}^{n+1}\|_{L^2}^2+\|R^{n+1}_\eta\|_{L^2}^2).\notag
\end{align}
where we also use the following relation:
 \begin{align*}
		&\big( \frac{r(t_{n+1})}{\sqrt{E_1(\eta(t_{n+1})) + C_1}} - \frac{r^{n+1}}{\sqrt{E_1(\eta^n) + C_1}} \big) \\
		&= \frac{-\sqrt{E_1(\eta^n) + C_1} \, e_r^{n+1} + r^{n+1} \big( \sqrt{E_1(\eta^n) + C_1} - \sqrt{E_1(\eta(t_{n+1})) + C_1} \big)}{\sqrt{E_1(\eta(t_{n+1})) + C_1} \cdot \sqrt{E_1(\eta^n) + C_1}}.\notag
\end{align*}

By using H\"older inequality and Young inequality, we can deduce
\begin{align}
	( \lambda e_{\eta}^{n+1}- \chi_0 e_c^{n+1}, e_{\mu \eta}^{n+1}) &\le C \|e_{\eta}^{n+1}\|_{L^2}^2+C \|e_c^{n+1}\|_{L^2}^2 + \epsilon_{\mu \eta} \|e_{\mu\eta}^{n+1}\|_{L^2}^2
	\\
	\frac{f'(\eta^n)}{\sqrt{E_1(\eta^n) + C_1}} (e_r^{n+1}, e_{\mu\eta}^{n+1}) &\le C |e_r^{n+1}|^2 + \epsilon_{\mu \eta} \|e_{\mu\eta}^{n+1}\|_{L^2}^2.
\end{align}

From (\ref{tumor-15}), we observe that
\begin{align}
(F_{\mu \eta}, e_{\mu\eta}^{n+1}) - (F_{\mu \eta}, \frac{e_{\eta}^{n+1} - e_{\eta}^n}{\tau}) &= (F_{\mu \eta}, e_{\mu\eta}^{n+1} - \Delta e_{u\eta}^{n+1} - F_\eta) \label{tumor-36}
	\\
	&= (\nabla F_{\mu \eta}, \nabla e_{\mu\eta}^{n+1})+(F_{\mu \eta}, e_{\mu\eta}^{n+1} - F_\eta).\notag
\end{align}

By (\ref{tumor-regularity-L2}) and (\ref{tumor-16}), we obtain
\begin{align}
	\|F_{\mu \eta}\|_{L^2} \le& C \|e_{\eta}^n\|_{L^2} + C \|\eta_{t}\|_{L^2} \tau,\\
	\|\nabla F_{\mu \eta}\|_{L^2} \le &C \|e_{\eta}^n\|_{L^2} + C \|\nabla e_{\eta}^n\|_{L^2} + C \|\eta_{t}\|_{L^2} \tau.
\end{align}

Therefore,  it follows that
\begin{align}
	(\nabla F_{\mu \eta}, \nabla e_{u\eta}^{n+1})
	\le C \|e_{\eta}^n\|_{L^2}^2
	+ C \|\nabla e_{\eta}^n\|_{L^2}^2
	+ C \|\eta_{t}\|_{L^2}^2 \tau^2
	+ \epsilon_{\mu\eta} \|\nabla e_{u\eta}^{n+1}\|_{L^2}^2 .
\end{align}

Similarly to the estimates in (\ref{tumor-A1}) and (\ref{tumor-36}), we can derive
\begin{align}
	&\quad	(F_{\mu \eta}, e_{\mu\eta}^{n+1} - F_\eta) -\chi_0 \big( e_c^{n+1}, \frac{e_\eta^{n+1} - e_\eta^n}{\tau} \big)\\
	&\leq C\|\nabla e_c^{n+1} \|_{L^2}^2+ C\|e^{n+1}_c\|^2_{L^2}+C \tau^2 + \epsilon_{\mu\eta} \|e_{\mu\eta}^{n+1}\|_{L^2}^2 + \epsilon_{\mu c} \|\nabla e_{\mu c}^{n+1}\|_{L^2}^2 \notag
	\\
	&\quad+ C \left( |e_r^{n+1}|^2 + \|e_{\eta}^n\|_{L^2}^2 + \|e_{\u}^n\|_{L^2}^2 + \|\nabla e_{\eta}^n\|_{L^2}^2 +\|R^{n+1}_\eta\|_{L^2}^2 \right) \notag
\end{align}

Substituting the above inequalities into (\ref{tumor-19}), we obtain (\ref{tumor-eta-results}).
	\end{proof}

			
\begin{lemma}
	Under the assumption (\ref{tumor-regularity}), we have
\begin{align}
		&\frac{1}{\tau}(|e^{n+1}_r|^2-|e^{n}_r|^2+|e^{n+1}_r-e^{n}_r|^2)\label{tumor-r-results}\\
		\leq& \frac{e^{n+1}_r}{\sqrt{E_1(\eta^n) + C_1}} \big( f'(\eta^n), \frac{e_\eta^{n+1} - e_\eta^n}{\tau} \big)+\epsilon_{\mu \eta}\|\nabla e^{n+1}_{\mu \eta}\|^2_{L^2} 
		+\frac{1}{4}\|e^n_{\mu \eta}\|^2_{L^2}\notag\\
		&+C \tau^2+C\tau ^2 \|\nabla (e^{n+1}_p-e^n_p)\|^2_{L^2}+\epsilon_{\mu c}\|\nabla e^{n+1}_{\mu c}\|^2_{L^2} 
		+\frac{1}{4}\|e^n_{\mu c}\|^2_{L^2}\notag\\
		&+C(|e^{n+1}_r|^2+C \|e^{n+1}_\u\|^2_{L^2}
		+ \|e^n_\u\|^2_{L^2}+|R^{n+1}_r|^2+C\|e^n_\eta\|^2).\notag
\end{align}
\end{lemma}
\begin{proof}
Equation (\ref{tumor-continuous3}) at $t=t_{n+1}$ satisfies
\begin{align}\label{tumor-continuous3-1}
	& \frac{r(t_{n+1}) - r(t_n)}{\tau} =  \frac{1}{2\sqrt{E_1(\eta(t_{n+1})) + C_1}} \int f'(\eta(t_{n+1})) \frac{\partial \eta}{\partial t}(t_{n+1}) dx+R_r^{n+1},
\end{align}
where
$
		R_r^{n+1}=\frac{r(t_{n+1}) - r(t_n)}{\tau} - r_t(t_{n+1}).
$

By the assumption (\ref{tumor-regularity}),we have
\begin{align}\label{tumor-45}
		|R_r^{n+1}|\leq | \frac{1}{\tau} \int_{t_n}^{t_{n+1}} (t_n - t) \frac{\partial^2 r}{\partial t^2} dt | \leq C \tau^{\frac{1}{2}}.
\end{align}

Subtracting (\ref{tumor-continuous3-1}) from (\ref{tumor-algorithm3}) 
and taking the inner product with $2e_r^{n+1}$, we obtain
%
\begin{align}
&\frac{1}{\tau}(|e^{n+1}_r|^2-|e^{n}_r|^2+|e^{n+1}_r-e^{n}_r|^2)\label{tumor-r-error-equation}\\
=&-2 e^{n+1}_r R_r^{n+1} + \frac{e^{n+1}_r}{\sqrt{E_1(\eta^n) + C_1}} \big( f'(\eta^n), \frac{e_\eta^{n+1} - e_\eta^n}{\tau} \big)\notag \\
&- \frac{e^{n+1}_r}{\sqrt{E_1(\eta^n) + C_1}} ( f'(\eta^n), \frac{\partial \eta}{\partial t}(t_{n+1}) - \frac{\eta(t_{n+1}) - \eta(t_n)}{\tau}  ) \notag\\
& + \frac{e^{n+1}_r}{\sqrt{E_1(\eta^n) + C_1}} \big( f'(\eta^n) - f'(\eta(t_{n+1})), \frac{\partial \eta}{\partial t}(t_{n+1}) \big)\notag \\
& + \big( \frac{e^{n+1}_r}{\sqrt{E_1(\eta^n) + C_1}} - \frac{e^{n+1}_r}{\sqrt{E_1(\eta(t_{n+1})) + C_1}} \big) ( f'(\eta(t_{n+1})), \frac{\partial \eta}{\partial t}(t_{n+1}) ) \notag\\
&+ \frac{e^{n+1}_r}{\sqrt{E_1(\eta^n) + C_1}} \big( ( \mu_\eta^{n+1}, \u^n \cdot \nabla \eta^n )-(\mu_\eta^n, \u^n \cdot \nabla \eta^n )\big) \notag\\
&+ \frac{e^{n+1}_r}{\sqrt{E_1(\eta^n) + C_1}} \big(( \u^n, \mu_\eta^n \nabla \eta^n)- (\tilde{\u}^{n+1}, \mu_\eta^n \nabla \eta^n )\big)\notag\\
&+ \frac{e^{n+1}_r}{\sqrt{E_1(\eta^n) + C_1}} \big( ( \mu_c^{n+1}, \u^n \cdot \nabla c^n )-(\mu_c^n, \u^n \cdot \nabla c^n )\big) \notag\\
&+ \frac{e^{n+1}_r}{\sqrt{E_1(\eta^n) + C_1}} \big(( \u^n, \mu_c^n \nabla c^n)- (\tilde{\u}^{n+1}, \mu_c^n \nabla c^n )\big)\notag\\
:=&\sum_{i=1}^{9}A_{i}.\notag
\end{align}

Using (\ref{tumor-45}), the H\"{o}lder inequality and Young's inequality, we have
\begin{align}
	&A_{1}+A_{3}+A_{4}+A_{5}\\
	\leq& C |R^{n+1}_r|^2+C|e^{n+1}_r|^2+C\tau^2+C\|e^n_\eta\|_{L^2}^2.\notag
\end{align}

It follows from (\ref{tumor-algorithm8}) that
\begin{align}
			\tilde{e}^{n+1}_\u
		=e^{n+1}_\u+\tau \nabla (e^{n+1}_p-e^n_p)+\tau \nabla (p(t_{n+1})-p(t_n)).
\end{align}
we can estimate $A_{6}+A_{7}$ by
\begin{align}
	A_{6}+A_{7}
	\leq& C|e^{n+1}_r| \|\mu_\eta^{n+1}-\mu_\eta^{n}\|_{L^2}\|\u^n\|_{L^4}\|\nabla \eta^n\|_{L^4}\notag\\
	&+C|e^{n+1}_r| \| \u^n-\tilde{\u}^{n+1}\|_{L^2}\|\mu_\eta^n\|_{L^3}\|\nabla \eta^n\|_{L^6}\notag\\
	\leq &C |e^{n+1}_r|^2+C \|e^{n+1}_\u\|^2_{L^2}+C \|e^n_\u\|^2_{L^2}+\epsilon_{\mu \eta}\|\nabla e^{n+1}_{\mu \eta}\|^2_{L^2} +\frac{1}{4}\|e^n_{\mu \eta}\|^2_{L^2} \notag\\
	&+C (\| \mu_{\eta t}\|^2_{H^1}+ \| \u_t \|^2_{L^2}) \tau^2+C\tau ^2 \|\nabla (e^{n+1}_p-e^n_p)\|^2_{L^2}
	+C \| p_t\|^2_{H^1} \tau^4.\notag
\end{align}
where we also use (\ref{tumor-eta-h2-induction}).

Similarly
\begin{align}
	A_{8}+A_{9}
	\leq &C |e^{n+1}_r|^2+C \|e^{n+1}_\u\|^2_{L^2}+C \|e^n_\u\|^2_{L^2}+\epsilon_{\mu c}\|\nabla e^{n+1}_{\mu c}\|^2_{L^2} +\frac{1}{4}\|e^n_{\mu c}\|^2_{L^2}\notag\\
	&+C (\| \mu_{c t}\|^2_{H^1}+ \| \u_t \|^2_{L^2}) \tau^2+C\tau ^2 \|\nabla (e^{n+1}_p-e^n_p)\|^2_{L^2}
	+C \| p_t\|^2_{H^1} \tau^4.\notag
\end{align}

Substituting the above inequalities into (\ref{tumor-r-error-equation}), we get (\ref{tumor-r-results}).
\end{proof}
\begin{lemma}
	Under the assumption (\ref{tumor-regularity}), we have
\begin{align}
		&\big( \frac{1+ \delta^{-1}}{2\tau} \big) \left( \| e_c^{n+1} \|_{L^2}^2 - \| e_c^n \|_{L^2}^2 + \| e_c^{n+1} - e_c^n \|_{L^2}^2 \right)\label{tumor-c-results}\\ 
		&+ \delta^{-1} \| \nabla e_c^{n+1} \|_{L^2}^2 + \| e_{\mu c}^{n+1} \|_{L^2}^2+  \| \nabla e_{\mu c}^{n+1} \|_{L^2}^2  +\|\sqrt{P(\eta^n)}e^{n+1}_{\mu c}\|^2_{L^2} \notag\\
		\leq &C \big(\| R_c^{n+1} \|_{L^2}^2 +\|e^{n+1}_c\|^2_{L^2} +\|e^{n}_\eta\|^2_{L^2}+\|e^{n}_c\|^2_{L^2}+\|\nabla e^{n}_\eta\|^2_{L^2}+\|e^n_\u\|^2_{L^2}+| e_r^{n+1} |^2\big)\notag\\
		&+C\tau^2+\epsilon_{\mu \eta} \| \nabla e^{n+1}_{\mu \eta}\|^2_{L^2}+\epsilon_{\mu c} \| \nabla e^{n+1}_{\mu c}\|^2_{L^2}.\notag
\end{align}
\end{lemma}
\begin{proof}
Equation (\ref{tumor-continuous4})--(\ref{tumor-continuous5}) at $t=t_{n+1}$ satisfy
\begin{align}
		&\frac{c(t_{n+1}) - c(t_n)}{\tau} +\frac{r(t_{n+1})}{\sqrt{E_1(\eta(t_{n+1})) + C_1}} \u(t_{n+1}) \cdot \nabla c(t_{n+1})- \Delta\mu_c(t_{n+1})\label{tumor-continuous4-1}\\
		&- P(\eta(t_n)) (\mu_\eta(t_{n+1}) - \mu_c(t_{n+1})) =R_c^{n+1}.\notag\\
	&\mu_c(t_{n+1})= \delta^{-1} c(t_{n+1}) -\chi_0 \eta(t_{n+1})\label{tumor-continuous5-1}.
\end{align}
where
\begin{align}
	R_c^{n+1}&=\frac{c(t_{n+1}) - c(t_n)}{\tau} - c_t(t_{n+1}) \\
	&-P(\eta(t_n)) (\mu_\eta(t_{n+1}) - \mu_c(t_{n+1}))+P(\eta(t_{n+1})) (\mu_\eta(t_{n+1}) - \mu_c(t_{n+1})).\notag
\end{align}

By the assumption (\ref{tumor-regularity}),we have
\begin{align*}
		\|R_c^{n+1}\|_{L^2}&\leq \| \frac{1}{\tau} \int_{t_n}^{t_{n+1}} (t_n - t) \frac{\partial^2 c}{\partial t^2} dt \|_{L^2} +C \| \eta(t_{n+1}) - \eta(t_n) \|_{L^4} (\| \mu_\eta \|_{L^4} + \| \mu_c \|_{L^4}) \notag\\\
		&\leq C \tau^{\frac{1}{2}}.\notag
\end{align*}

Subtracting (\ref{tumor-continuous4-1})--(\ref{tumor-continuous5-1}) from (\ref{tumor-algorithm4})--(\ref{tumor-algorithm5}), we have the following error equations
\begin{align}
			\frac{e_c^{n+1} - e_c^n}{\tau} & -\Delta e^{n+1}_{\mu c}+\frac{r(t_{n+1})}{\sqrt{E_1(\eta(t_{n+1}))+C_1}}  \nabla \cdot \big( \u(t_{n+1})c(t_{n+1}) - \u(t_n)c(t_n) \big)\notag \\
			& \quad+\big( \frac{r(t_{n+1})}{\sqrt{E_1(\eta(t_{n+1}))+C_1}} - \frac{r^{n+1}}{\sqrt{E_1(\eta^n)+C_1}}\big)  \nabla \cdot (\u(t_n)c(t_n))\notag \\
			&\quad- \frac{r^{n+1}}{\sqrt{E_1(\eta^n)+C_1}} \nabla \cdot (\u(t_n)e_c^n)
			- \frac{r^{n+1}}{\sqrt{E_1(\eta^n)+C_1}} \nabla \cdot (e_\u^n c^n)\label{tumor-c-error-eqaution}\\
			&\quad- P(\eta^n) (e_{\mu\eta}^{n+1} - e_{\mu c}^{n+1})- \big(P(\eta^n) - P(\eta(t_n))\big) (\mu_\eta(t_{n+1}) - \mu_c(t_{n+1})) \notag\\
            &= -R_c^{n+1}.\notag\\
		e_{\mu c}^{n+1} &= \delta^{-1} e_c^{n+1} - \chi_0 e_\eta^n + \chi_0 (\eta(t_{n+1}) - \eta(t_n)).\label{tumor-muc-error-equation}
\end{align}

Taking the inner product of \eqref{tumor-c-error-eqaution} with $e^{n+1}_{\mu c}$, and \eqref{tumor-muc-error-equation} with $e^{n+1}_{\mu c}$ and $\frac{e^{n+1}_c - e^n_c}{\tau}$, respectively. 
By coupling \eqref{tumor-c-error-eqaution} and \eqref{tumor-muc-error-equation}, we further test the resulting system against $e^{n+1}_c$.
Substituting these error equations, we have the following error equation
\begin{align}
		&\quad\big( \frac{1 + \delta^{-1}}{2\tau} \big) \left( \| e_c^{n+1} \|_{L^2}^2 - \| e_c^n \|_{L^2}^2 + \| e_c^{n+1} - e_c^n \|_{L^2}^2 \right) \label{tumor-c-error-identify}\\
		&\quad+ \delta^{-1} \| \nabla e_c^{n+1} \|_{L^2}^2 + \| e_{\mu c}^{n+1} \|_{L^2}^2+  \| \nabla e_{\mu c}^{n+1} \|_{L^2}^2  \notag\\
		&=\frac{r(t_{n+1})}{\sqrt{E_1(\eta(t_{n+1}))+C_1}} \big( \nabla \cdot (\u(t_{n+1})c(t_{n+1}) - \u(t_n)c(t_n)) , e^{n+1}_{\mu c}+e^{n+1}_c\big)\notag \\
		& \quad+ \big(\frac{r(t_{n+1})}{\sqrt{E_1(\eta(t_{n+1}))+C_1}} - \frac{r^{n+1}}{\sqrt{E_1(\eta^n)+C_1}}\big) \big( \nabla \cdot (\u(t_n)c(t_n)),e^{n+1}_{\mu c}+e^{n+1}_c \big) \notag\\
		&\quad- \frac{r^{n+1}}{\sqrt{E_1(\eta^n)+C_1}} \big( \nabla \cdot (\u(t_n)e_c^n),e^{n+1}_{\mu c}+e^{n+1}_c \big)\notag\\
		&\quad- \frac{r^{n+1}}{\sqrt{E_1(\eta^n)+C_1}} \big(\nabla \cdot (e_\u^n c^n),e^{n+1}_{\mu c}+e^{n+1}_c \big)\notag\\
		& \quad+ (P(\eta^n)(e_{\mu \eta}^{n+1} - e_{\mu c}^{n+1}), e_{\mu c}^{n+1})- (R_c^{n+1}, e_{\mu c}^{n+1}+e^{n+1}_c)  \notag\\
		&\quad + ((P(\eta^n) - P(\eta(t_n)))(\mu_\eta(t_{n+1}) - \mu_c(t_{n+1})), e_{\mu c}^{n+1}+e^{n+1}_c) \notag\\
		&\quad + \delta^{-1} (e_c^{n+1}, e_{\mu c}^{n+1}) - \chi_0 (e_\eta^n, e_{\mu c}^{n+1}) + \chi_0 (\eta(t_{n+1}) - \eta(t_n), e_{\mu c}^{n+1})\notag \\
		&\quad +\chi_0 \big( e_\eta^n, \frac{e_c^{n+1} - e_c^n}{\tau} \big) - \chi_0 \big( \eta(t_{n+1}) - \eta(t_n), \frac{e_c^{n+1} - e_c^n}{\tau} \big) \notag\\
		&\quad + \chi_0 (\nabla e_\eta^n, \nabla e_c^{n+1}) - \chi_0 (\nabla (\eta(t_{n+1}) - \eta(t_n)), \nabla e_c^{n+1})\notag \\
		&\quad + (P(\eta^n)(e_{\mu \eta}^{n+1} - e_{\mu c}^{n+1}), e_c^{n+1}) \notag\\
	&:=\sum_{i=1}^{15}\Lambda_{i}.\notag
\end{align}

Similar to (\ref{tumor-A1}), we get
\begin{align}
	\Lambda_1 & \le C \| e_{\mu c}^{n+1}+e^{n+1}_c \|_{L^2} \left( \| \u(t_{n+1}) \|_{L^\infty} \| c_t \|_{L^2} \tau + \| c(t_n) \|_{L^\infty} \| \u_t \|_{L^2} \tau \right) \\
	& \le \epsilon_{\mu c}  \| \nabla e_{\mu c}^{n+1} \|_{L^2}^2 +\epsilon_{c}  \| \nabla e_{c}^{n+1} \|_{L^2}^2 + C \tau^2,\notag\\
	\Lambda_2 & \le C \| e_{\mu c}^{n+1}+e^{n+1}_c \|_{L^2} | e_r^{n+1} | + C \| e_{\mu c}^{n+1}+e^{n+1}_c  \|_{L^2} \left( \| e_\eta^n \|_{L^2} + \| \eta_t \|_{L^2} \tau \right) \\
	& \le \epsilon_{\mu c}\| \nabla e_{\mu c}^{n+1} \|_{L^2}^2 +\epsilon_{c}  \| \nabla e_{c}^{n+1} \|_{L^2}^2 + C | e_r^{n+1} |^2 + C \| e_\eta^n \|_{L^2}^2 + C \tau^2.\notag
\end{align}

By using integration by parts, we can deduce that
\begin{align}
	\Lambda_3 +\Lambda_4 & \le C \| \nabla( e_{\mu c}^{n+1}+e^{n+1}_c) \|_{L^2} \| \u(t_n) \|_{L^\infty} \| e_c^n \|_{L^2}\\
    &\quad+ C \| \nabla( e_{\mu c}^{n+1} +e^{n+1}_c)\|_{L^2} \| e_\u^n \|_{L^2} \notag\\
	& \le \epsilon_{\mu c} \| \nabla e_{\mu c}^{n+1} \|_{L^2}^2 +\epsilon_{c}  \| \nabla e_{c}^{n+1} \|_{L^2}^2 + C \| e_c^n \|_{L^2}^2 + C \| e_\u^n \|_{L^2}^2.\notag
\end{align}

By (\ref{tumor-regularity-L2}) and (\ref{tumor-muc-error-equation}), together with the H\"{o}lder and Young inequalities, we obtain
\begin{align}
	\Lambda_{5}=&-\| \sqrt{P(\eta^n)}e^{n+1}_{\mu c}\|^2_{L^2}+\big(P(\eta^n) e^{n+1}_{\mu \eta}, e^{n+1}_{\mu c}\big)\label{tumor-50}\\
	\leq & -\| \sqrt{P(\eta^n)}e^{n+1}_{\mu c}\|^2_{L^2}+C \| e_c^{n+1}\|_{L^2}^2+C\| e^n_\eta\|^2_{L^2}\notag\\
    &+C \| \eta_t\|_{H^1}^2 \tau^2 +\epsilon_{\mu \eta} \| e^{n+1}_{\mu\eta}\|^2_{L^2}.\notag
	\end{align}
	and
	\begin{align}
	&\Lambda_{6}+\Lambda_{7}+\Lambda_{8}+\Lambda_{9}+\Lambda_{10}\\
	\leq &\epsilon_{c} \|\nabla e_c^{n+1}\|_{L^2}^2 + C \|e_\u^n\|_{L^2}^2+ C \| \nabla e_\eta^n\|_{L^2}^2 + \epsilon_{\mu c} \|\nabla e_{\mu c}^{n+1}\|_{L^2}^2\notag\\
	&+C \|R_c^{n+1}\|_{L^2}^2 +C \|e_c^{n+1}\|_{L^2}^2+C \|e_\eta^{n}\|_{L^2}^2 + C \|\eta_t\|_{L^2}^2 \tau^2.\notag
\end{align}

From (\ref{tumor-c-error-eqaution}),  In terms of (\ref{tumor-regularity}), (\ref{tumor-p-eta-infty}) and integration by parts, adapting the same technique for $\Lambda_{i}, i=1,2,3,4$, we can deduce that
\begin{align}
\Lambda_{11}+	\Lambda_{12}
		\leq	& C\|e^n_\eta\|^2_{L^2}+C\|e^n_\u\|^2_{L^2}+C\| \nabla e_\eta^n\|_{L^2}^2+C\|R^{n+1}_c\|^2_{L^2}\notag\\
		&+\epsilon_c \|\nabla e^{n+1}_c\|^2_{L^2}+\epsilon_{\mu c}\|\nabla e^{n+1}_{\mu c}\|^2_{L^2}+\epsilon_{\mu \eta}\|\nabla e^{n+1}_{\mu \eta}\|^2_{L^2}\notag\\
		&+C \| \eta_t \|_{H^1}^2 \tau^2+C | e_r^{n+1} |^2+ C \| e_c^n \|_{L^2}^2
\end{align}

By using (\ref{tumor-p-eta-infty}), H\"{o}lder inequality and Young inequality, one has
\begin{align}
	&\Lambda_{13} +	\Lambda_{14}+\Lambda_{15}\\
	 \leq &C \|\nabla e_\eta^n\|_{L^2}^2 + \epsilon_c \|\nabla e_c^{n+1}\|_{L^2}^2 + C \|\eta_t\|_{H^1}^2 \tau^2+ C\|e^n_\eta\|^2_{L^2}\notag\\
	 &+C\|e^{n+1}_c\|^2_{L^2}+\epsilon_{\mu \eta} \|e_{\mu \eta}^{n+1}\|_{L^2}^2 + \epsilon_{\mu c} \|e_{\mu c}^{n+1}\|_{L^2}^2.\notag
\end{align}

Substituting the above inequalities into (\ref{tumor-c-error-identify}), we arrive at (\ref{tumor-c-results}).
%
\end{proof}

By the same technique in \cite{shenjie2022}, we can get the following two lemmas.
\begin{lemma}
	Under the assumption (\ref{tumor-regularity}), we have
\begin{align}
		&\frac{1}{2\tau}(\| e^{n+1}_\u\|^2_{L^2}-\|e^n_\u\|^2_{L^2} + \| \tilde{e}^{n+1}_\u-e^n_\u\|^2_{L^2})\label{tumor-u-results}\\
		&+\| \nabla \tilde{e}^{n+1}_\u\|^2_{L^2}+\frac{\tau}{2} (\|\nabla e_{p}^{n+1}\|^{2} - \|\nabla e_{p}^{n}\|^{2})\notag\\
		\leq&-\exp \left(\frac{t_{n+1}}{T}\right) e_{q}^{n+1}\left(\u^{n} \cdot \nabla \u^{n}, \tilde{e}_{\mathbf{u}}^{n+1}\right)+  \epsilon_\u\left\|\nabla \tilde{e}_{\mathbf{u}}^{n+1}\right\|_{L^{2}}^{2}+C \tau^{2}\notag\\
		&+C\left\|e_{\mathbf{u}}^{n}\right\|_{L^{2}}^{2} +C\left\|\nabla e_{\mathbf{u}}^{n}\right\|_{L^{2}}^{2}\left\|e_{\mathbf{u}}^{n}\right\|_{L^{2}}^{2}+\frac{1}{4}\left\| e_{\mu \eta}^{n}\right\|_{L^{2}}^{2} +\frac{1}{2}\left\|  \nabla e_{\mu \eta}^{n}\right\|_{L^{2}}^{2} \notag\\
		&+\frac{1}{4}\left\| e_{\mu c}^{n}\right\|_{L^{2}}^{2} +\frac{1}{2}\left\|  \nabla e_{\mu c}^{n}\right\|_{L^{2}}^{2}  +C\|\nabla e_{\eta}^{n}\|_{L^{2}}^{2} \notag\\
		&+ C|e_{r}^{n+1}|^{2}+C\tau^2(\|\nabla e^{n+1}_p\|^2_{L^2}+\|\nabla e_{p}^{n}\|^2_{L^2})+C\| R^{n+1}_\u\|^2_{L^2}.\notag
\end{align}
\end{lemma}
\begin{lemma}
	Under the assumption (\ref{tumor-regularity}), we have
\begin{align}
		& \frac{1}{2\tau} \left(|e_{q}^{n+1}|^{2} - |e_{q}^{n}|^{2} + |e_{q}^{n+1}-e_{q}^{n}|^{2}\right) + \frac{1}{T} |e_{q}^{n+1}|^{2} \label{tumor-q-results}\\
			\leq & \exp \left(\frac{t_{n+1}}{T}\right) e_{q}^{n+1} (\u^{n} \cdot \nabla \u^{n}, \tilde{e}_{\mathbf{u}}^{n+1}) + C\|e_{\mathbf{u}}^{n}\|_{L^{2}}^{2} + C \tau^{2} + C|R_{q}^{n+1}|^{2}.\notag
	\end{align}
\end{lemma}

\subsection{The proof of Theorem \ref{tumor-main-theorem}}
\begin{proof}
Since
\begin{align}
    \tau \sum^{N-1}_{n=1} \big(\|R^{n+1}_\eta\|_{L^2}^2+|R^{n+1}_r|^2+\|R^{n+1}_c\|_{L^2}^2+\|R^{n+1}_\u\|_{L^2}^2+|R^{n+1}_q|^2\big)\leq C\tau^2.
\end{align}

Applying the technique used in \eqref{tumor-50} to the term $\big(P(\eta^n) e^{n+1}_{\mu \eta}, e^{n+1}_{\mu c}\big)$ in \eqref{tumor-eta-results}, we can combine the outcome with the primary estimates \eqref{tumor-eta-results}, \eqref{tumor-r-results}, \eqref{tumor-c-results}, \eqref{tumor-u-results}, and \eqref{tumor-q-results} to obtain
\begin{align}
		&\frac{\lambda + \epsilon^2}{2\tau} \left( \|e_{\eta}^{n+1}\|_{L^2}^2 - \|e_{\eta}^n\|_{L^2}^2 + \|e_{\eta}^{n+1} - e_{\eta}^n\|_{L^2}^2 \right) \\
		&+ \frac{\epsilon^2}{2\tau} \left( \|\nabla e_{\eta}^{n+1}\|_{L^2}^2 - \|\nabla e_{\eta}^n\|_{L^2}^2 + \|\nabla(e_{\eta}^{n+1} - e_{\eta}^n)\|_{L^2}^2 \right) \notag\\
		&+\frac{1}{\tau}(|e^{n+1}_r|^2-|e^{n}_r|^2+|e^{n+1}_r-e^{n}_r|^2)\notag\\
		&+\left( \frac{1 + \delta^{-1}}{2\tau} \right) \left( \| e_c^{n+1} \|_{L^2}^2 - \| e_c^n \|_{L^2}^2 + \| e_c^{n+1} - e_c^n \|_{L^2}^2 \right)\notag\\
		&+\frac{1}{2\tau}(\| e^{n+1}_\u\|^2_{L^2}-\|e^n_\u\|^2_{L^2} + \| \tilde{e}^{n+1}_\u-e^n_\u\|^2_{L^2})\notag\\
			& +\frac{1}{2\tau} \left(|e_{q}^{n+1}|^{2} - |e_{q}^{n}|^{2} + |e_{q}^{n+1}-e_{q}^{n}|^{2}\right)\notag\\
			& +\frac{\tau}{2} (\|\nabla e_{p}^{n+1}\|^{2} - \|\nabla e_{p}^{n}\|^{2}) +\|\sqrt{P(\eta^n)}e^{n+1}_{\mu \eta} \|^2_{L^2}+ \|\sqrt{P(\eta^n)}e^{n+1}_{\mu c} \|^2_{L^2}\notag\\
			&+ \delta^{-1} \| \nabla e_c^{n+1} \|_{L^2}^2 + \| e_{\mu c}^{n+1} \|_{L^2}^2-\frac{1}{2} \| e_{\mu c}^{n} \|_{L^2}^2+  \| \nabla e_{\mu c}^{n+1} \|_{L^2}^2-\frac{1}{2}\| \nabla e_{\mu c}^{n} \|_{L^2}^2 \notag\\
			&+ \|e_{\mu \eta}^{n+1}\|_{L^2}^2 -\frac{1}{2} \|e_{\mu \eta}^{n}\|_{L^2}^2+ \|\nabla e_{\mu \eta}^{n+1}\|_{L^2}^2-\frac{1}{2}\|\nabla e_{\mu \eta}^{n}\|_{L^2}^2+ \frac{1}{T} |e_{q}^{n+1}|^{2}+\| \nabla \tilde{e}^{n+1}_\u\|^2_{L^2}\notag\\
			\leq&C (|e_r^{n+1}|^2 +\|e_c^{n+1}\|_{L^2}^2+ \epsilon^2 \|\nabla e_{\eta}^{n+1}\|_{L^2}^2+ \|e_\u^n\|_{L^2}^2 \notag\\
			&+ \|\nabla e_{\eta}^n\|_{L^2}^2 +\|e_{\eta}^n\|_{L^2}^2)+C\left\|\nabla e_{\mathbf{u}}^{n}\right\|_{L^{2}}^{2}\left\|e_{\mathbf{u}}^{n}\right\|_{L^{2}}^{2}+C\tau^2(\|\nabla e^{n+1}_p\|^2_{L^2}+\|\nabla e_{p}^{n}\|^2_{L^2})\notag\\&+C\big(\|R^{n+1}_\eta\|_{L^2}^2+|R^{n+1}_r|^2+\|R^{n+1}_c\|_{L^2}^2+\|R^{n+1}_\u\|_{L^2}^2+|R^{n+1}_q|^2\big)\notag\\
			&+ \epsilon_{\mu \eta} \|\nabla e_{u\eta}^{n+1}\|_{L^2}^2 + \epsilon_{\mu c} \|\nabla e_{uc}^{n+1}\|_{L^2}^2+\epsilon_\u\left\|\nabla \tilde{e}_{\mathbf{u}}^{n+1}\right\|_{L^{2}}^{2}+\epsilon_{c} \|\nabla e_c^{n+1}\|_{L^2}^2 +C \tau^2.\notag
\end{align}

Multiplying it by $\tau$, summing up from $n=0$ to $N-1$ and  by using Gronwall  inequalities in Lemma 5.1 of \cite{heywood1990}, for sufficiently small $\epsilon_{\mu \eta}, \epsilon_{\mu c}, \epsilon_{\u}, \epsilon_c$, we arrive at (\ref{tumor-main-results}). The proof of Theorem  \ref{tumor-main-theorem} is done.
\end{proof}

Under this induction hypothesis \eqref{tumor-eta-h2-induction}, Theorem  \ref{tumor-main-theorem} imply the error estimate
up to the time level ${n+1}$. We then use this error estimate to prove
$$
\|\nabla\mathbf u^{n+1}\|_{L^2}
+\|\Delta\eta^{n+1}\|_{L^2}
+\|\Delta c^{n+1}\|_{L^2}
\leq C.
$$
This closes the induction.
To complete the mathematical induction for $\| \nabla \u^{n+1}\|_{L^2}$, from (\ref{tumor-main-results}), we can derive 
$
	\tau  \sum_{n=0}^{N-1} \|\nabla \tilde{e}_{\mathbf{u}}^{n+1}\|^2 \leq C\tau^2,
$
which implies that
\begin{align}\label{tumor-u-h1}
	\|\nabla \mathbf{u}^{n+1}\|_{L^2} \leq C\|\nabla \tilde{\mathbf{u}}^{n+1}\|_{L^2} \leq C \left( \tau^{1/2} + \|\nabla \mathbf{u}(t_{n+1})\|_{L^2} \right)\leq C.
\end{align}
the above inequality holds thanks to the fact that \cite{temam2024,shenjie2021},
\begin{align}
\|\mathbf{u}^{n+1}\|_{\mathbf{H}^1(\Omega)} = \|P_H \tilde{\mathbf{u}}^{n+1}\|_{\mathbf{H}^1(\Omega)} \leq C(\Omega) \|\tilde{\mathbf{u}}^{n+1}\|_{\mathbf{H}^1(\Omega)}.
\end{align}

The corresponding estimates for  $\|\Delta \eta^{n+1}\|_{L^2}$ and $\|\Delta c^{n+1}\|_{L^2}$ will be established by the following theorem.
 \begin{theorem}\label{tumor-eta-h2-theorem}
	Under the assumption (\ref{tumor-regularity}), there exists a positive $C_2$ independent of $\tau$ such that
	\begin{align}
			& \left\| \Delta \eta^{N} \right\|_{L^2}^2 +\lambda \tau \sum_{n=0}^{N-1} \left\| \nabla \Delta \eta^{n+1} \right\|_{L^2}^2 + \epsilon^2 \tau \sum_{n=0}^{N-1} \left\| \Delta^2 \eta^{n+1} \right\|_{L^2}^2  \label{tumor-eta-h2}\\
			& +\left\| \nabla c^{N} \right\|_{L^2}^2 +\delta^{-1} \tau\sum_{n=0}^{N-1}  \left\| \Delta c^{n+1} \right\|_{L^2}^2  
			\leq C_2.\notag
	\end{align}
which implis $\|\Delta \eta^{n+1}\|^2_{L^2}\leq C.$
\end{theorem}
\begin{proof}
	We will give the proof by mathematical induction, assuming that 
	$
\| \Delta c^i \|_{L^2}+\| \Delta \eta^i \|_{L^2} \leq C, \, 0 \leq i \leq n.
	$
	According to the algorithm (\ref{tumor-algorithm1}), (\ref{tumor-algorithm2}), (\ref{tumor-algorithm4}), (\ref{tumor-algorithm5}), we have 
	\begin{align}
		& \frac{\eta^{n+1}-\eta^n}{\tau} + \frac{r^{n+1}}{\sqrt{E_1(\eta^n)+C_1}} \u^n \cdot \nabla \eta^n -  \lambda\Delta \eta^{n+1} + \epsilon^2 \Delta^2 \eta^{n+1} \label{tumor-22}\\
		&\quad+ \chi_0 \Delta c^{n+1} - \frac{r^{n+1}}{\sqrt{E_1(\eta^n)+C_1}} \Delta f'(\eta^n) + P(\eta^n) \big( \lambda \eta^{n+1} - \epsilon^2 \Delta \eta^{n+1}  \notag\\
		&\quad- \chi_0 c^{n+1}+ \frac{r^{n+1}}{\sqrt{E_1(\eta^n)+C_1}} f'(\eta^n) - \delta^{-1} c^{n+1} + \chi_0 \eta^n \big) = 0,\notag\\
		& \frac{c^{n+1}-c^n}{\tau} + \frac{r^{n+1}}{\sqrt{E_1(\eta^n)+C_1}} \u^n \cdot \nabla c^n- \delta^{-1} \Delta c^{n+1} + \chi_0 \Delta \eta^n - P(\eta^n) \big( \lambda \eta^{n+1}  \label{tumor-23}\\
		&\quad- \epsilon^2 \Delta \eta^{n+1}- \chi_0 c^{n+1} + \frac{r^{n+1}}{\sqrt{E_1(\eta^n)+C_1}} f'(\eta^n) - \delta^{-1} c^{n+1} + \chi_0 \eta^n \big) = 0.\notag
	\end{align}
	
	Test (\ref{tumor-22}) by $\Delta^2\eta^{n+1}$, we have
	\begin{align}
		& \frac{1}{2\tau}\big(\left\| \Delta \eta^{n+1} \right\|_{L^2}^2 - \left\| \Delta \eta^n \right\|_{L^2}^2 + \left\| \Delta \eta^{n+1} - \Delta \eta^n \right\|_{L^2}^2\big)\label{tumor-26}\\
		&
		\quad+ \lambda \left\| \nabla \Delta \eta^{n+1} \right\|_{L^2}^2 + \epsilon^2 \left\| \Delta^2 \eta^{n+1} \right\|_{L^2}^2 \notag \\
		= & - \frac{r^{n+1}}{\sqrt{E_1(\eta^n)+C_1}} (\u^n \cdot \nabla \eta^n, \Delta^2 \eta^{n+1}) - \chi_0 (\Delta c^{n+1}, \Delta^2 \eta^{n+1}) \notag\\
		&+ \frac{r^{n+1}}{\sqrt{E_1(\eta^n)+C_1}} (\Delta f'(\eta^n), \Delta^2 \eta^{n+1})-(\lambda P(\eta^n) \eta^{n+1},\Delta^2 \eta^{n+1} )\notag\\
		& +(\epsilon^2P(\eta^n) \Delta \eta^{n+1},\Delta^2 \eta^{n+1}) + (\chi_0P(\eta^n) c^{n+1}, \Delta^2 \eta^{n+1})  \notag\\
		&-  \frac{r^{n+1}}{\sqrt{E_1(\eta^n)+C_1}} (P(\eta^n)f'(\eta^n), \Delta^2 \eta^{n+1})+ (\delta^{-1} P(\eta^n) c^{n+1}, \Delta^2 \eta^{n+1}) \notag\\
		& - (\chi_0 P(\eta^n)  \eta^n, \Delta^2 \eta^{n+1})\notag\\
		:=& \sum_{i=1}^{9}\Xi_i. \notag
	\end{align}
	
	By noticing that $|r^{n+1}| \leq C$ and $E_1+C_1 >C_0$, taking use of $\| a\|_{L^4}\leq C \| a \|^{\frac{1}{2}}_{L^2}\|\nabla a \|^{\frac{1}{2}}_{L^2}$, we can deduce
	\begin{align}
	\Xi_1
		&\leq C \| \u^n\|^{\frac{1}{2}}_{L^2}\|\u^n\|_{H^1}^{\frac{1}{2}} \| \nabla \eta^n\|_{L^2}^{\frac{1}{2}} \| \nabla \eta^n\|_{H^1}^{\frac{1}{2}}\| \Delta^2 \eta^{n+1}\|_{L^2}\label{tumor-38}\\
		&\leq C\| \Delta \eta^n\|_{L^2}^2+ \epsilon_1 \| \Delta^2 \eta^{n+1}\|^2_{L^2}.\notag
	\end{align}
	
Adapt the same technique in Lemma 2.4 of \cite{shenjie2018} and using (\ref{tumor-p-assumption}), one has
	\begin{align}
	\Xi_3\leq& C\|\Delta f'(\eta^n)\|_{L^2} \| \Delta ^2 \eta^{n+1}\|_{L^2}\label{tumor-25}\\
		\leq &\epsilon_2 \| \Delta^2 \eta^{n+1}\|_{L^2}^2+\frac{\epsilon^2}{2}\| \Delta^2 \eta^n\|^2_{L^2}+C(M),\notag\\
		\Xi_7 \leq & C\| P(\eta^n)\|_{L^{\infty}} \|f'(\eta^n)\|_{L^2} \| \Delta^2 \eta^{n+1}\|_{L^2}\label{tumor-44}\\
		\leq& \epsilon_1\| \Delta^2 \eta^{n+1}\|_{L^2}^2+C.\notag
	\end{align}
	
	By using (\ref{tumor-regularity-L2}), (\ref{tumor-p-eta-infty}) H\"older inequality and Young inequality, we have
	\begin{align}
		&\Xi_4+\Xi_5+\Xi_6+\Xi_8+\Xi_9\label{tumor-41}\\
		\leq &C\| \Delta \eta^{n}\|_{L^2}^2+\epsilon_1 \| \Delta^2 \eta^{n+1}\|^2_{L^2}+C \| \Delta \eta^{n+1}\|_{L^2} ^2+C.\notag
	\end{align}


	As for $\Xi_2$, from (\ref{tumor-23}),  (\ref{tumor-main-results}), (\ref{tumor-regularity}), we have
	\begin{align}
		\Delta c^{n+1} &= \delta \frac{c^{n+1}-c^n}{\tau} + \delta \frac{r^{n+1}}{\sqrt{E_1(\eta^n)+C_1}} \u^n \cdot \nabla c^n+ \delta \chi_0 \Delta \eta^{n} \label{tumor-42}\\
		&\quad - \delta P(\eta^n) \big( \lambda \eta^{n+1} - \epsilon^2 \Delta \eta^{n+1} - \chi_0 c^{n+1} + \frac{r^{n+1}}{\sqrt{E_1(\eta^n) + C_1}} f'(\eta^n) - \delta^{-1} c^{n+1} + \chi_0 \eta^n \big).\notag
	\end{align}
	
	Furthermore, from (\ref{tumor-main-results}), we can deduce that
	\begin{align}\label{tumor-577}
		&- \chi_0 \delta \left(  \frac{c^{n+1}-c^n}{\tau}, \Delta^2 \eta^{n+1} \right) \\
		\leq &\epsilon_1 \|\Delta^2 \eta^{n+1}\|_{L^2}^2 + \frac{C}{\tau^2} \left( \|{e_c^{n+1}}\|_{L^2}^2 + \|{e_c^n}\|_{L^2}^2 \right) + C \|{c_t}\|_{L^2}^2 \notag\\
		\leq &\epsilon_1 \|\Delta^2 \eta^{n+1}\|_{L^2}^2 + C.\notag
	\end{align}

Similar to (\ref{tumor-38}), we get
	\begin{align}\label{tumor-43}
\delta \frac{r^{n+1}}{\sqrt{E_1(\eta^n)+C_1}} (\u^n \cdot \nabla c^n,\Delta^2 \eta^{n+1})
		\leq  \epsilon_1 \|{\Delta^2 \eta^{n+1}}\|_{L^2}^2 + C \|{\Delta c^{n}}\|_{L^2}^2.\notag
	\end{align}
	
	Combining  (\ref{tumor-44}) -- (\ref{tumor-577}), we have
	\begin{align}
	\Xi_2\leq C+\epsilon_1 \|{\Delta^2 \eta^{n+1}}\|_{L^2}^2 + \epsilon_2 \|{\Delta c^{n+1}}\|_{L^2}^2+C\| \Delta \eta^{n+1}\|_{L^2}^2+C\|\Delta \eta^n\|^2_{L^2}.
	\end{align}
	
	Test (\ref{tumor-23}) by $-\Delta c^{n+1}$, we get
	\begin{align}\label{tumor-24}
			&\frac{1}{2\tau} \big(\left\| \nabla c^{n+1} \right\|_{L^2}^2 - \left\| \nabla c^n \right\|_{L^2}^2 + \left\| \nabla (c^{n+1} - c^n) \right\|_{L^2}^2\big) + \delta^{-1} \left\| \Delta c^{n+1} \right\|_{L^2}^2 \\
			= &  \frac{r^{n+1}}{\sqrt{E_1(\eta^n)+C_1}} (\u^n \cdot \nabla c^n, \Delta c^{n+1}) + (\chi_0 \Delta \eta^{n}, \Delta c^{n+1}) -  (\lambda P(\eta^n) \eta^{n+1}, \Delta c^{n+1}) \notag\\
			& +  (\epsilon^2 P(\eta^n) \Delta \eta^{n+1}, \Delta c^{n+1}) +(\chi_0 P(\eta^n) c^{n+1},\Delta c^{n+1} )\notag\\
			&-   \frac{r^{n+1}}{\sqrt{E_1(\eta^n)+C_1}}\left( P(\eta^n) f'(\eta^n), \Delta c^{n+1} \right) +(\delta^{-1}P(\eta^n) c^{n+1},\Delta c^{n+1})\notag\\
			& -  (\chi_0 P(\eta^n) \eta^n, \Delta c^{n+1})\notag\\
			:=&\sum_{i=10}^{17}\Xi_i.\notag
	\end{align}
	
	Next we estimate the right hand of (\ref{tumor-24}). Similar to (\ref{tumor-38}), we get
	\begin{align}
		\Xi_{10}
			\leq \epsilon_2 \| \Delta c^{n+1}\|_{L^2}^2+C.\notag
	\end{align}
	
	By (\ref{tumor-regularity-L2}), (\ref{tumor-p-eta-infty}), H\"older inequality and Young inequality, we have
	\begin{align}
		&\Xi_{11}+	\Xi_{12}+	\Xi_{13}+	\Xi_{14}+	\Xi_{16}+	\Xi_{17}\\
			\leq &C \| \Delta \eta^{n+1} \|_{L^2}^2 +C \| \Delta \eta^{n} \|_{L^2}^2 + \epsilon_2 \| \Delta c^{n+1}\|^2_{L^2}+C.\notag
	\end{align}
	
	By the same technique in (\ref{tumor-44}), we have
	\begin{align}
\Xi_{15} \leq & \epsilon_2\| \Delta c^{n+1}\|_{L^2}^2+C.
	\end{align}
	
	Substituting $\Xi_i$ and Summing up (\ref{tumor-26}) and (\ref{tumor-24}), multiplying it by $\tau$, for sufficiently small $\epsilon_1,\epsilon_2$, we have
	\begin{align}
		\begin{aligned}
			& \left\| \Delta \eta^{n+1} \right\|_{L^2}^2 - \left\| \Delta \eta^n \right\|_{L^2}^2 + \left\| \Delta \eta^{n+1} - \Delta \eta^n \right\|_{L^2}^2 \\
			& +\left\| \nabla c^{n+1} \right\|_{L^2}^2 - \left\| \nabla c^n \right\|_{L^2}^2 + \left\| \nabla (c^{n+1} - c^n) \right\|_{L^2}^2 \\
			&+ \lambda \tau \left\| \nabla \Delta \eta^{n+1} \right\|_{L^2}^2 + \epsilon^2 \tau \left\| \Delta^2 \eta^{n+1} \right\|_{L^2}^2 -\frac{\epsilon^2}{2} \tau \| \Delta^2 \eta^n\|^2_{L^2} + \delta^{-1} \tau \left\| \Delta c^{n+1} \right\|_{L^2}^2  \\
			\leq &C \tau \| \Delta \eta^{n+1}\|^2_{L^2}+C \tau \| \Delta \eta^{n}\|^2_{L^2}+C\tau.
		\end{aligned}
	\end{align}

	Summing the above inequality over $n=0,1,2,\ldots,N-1$ and by using Gronwall  inequalities in Lemma 5.1 of \cite{heywood1990},
	we obtain \eqref{tumor-eta-h2}.
	This completes the proof of Theorem~\ref{tumor-eta-h2-theorem}.
\end{proof}

 Finally, from \eqref{tumor-42}, Theorem \ref{tumor-main-theorem} and Theorem~\ref{tumor-eta-h2-theorem}, and regularity of elliptic problem, we can deduce
 \begin{align*}
 \|\Delta c^{n+1}\|_{L^2}
 \leq &C \left( \|\frac{e_c^{n+1} - e_c^n}{\tau}\|_{L^2} + \|\frac{c(t_{n+1}) - c(t_n)}{\tau}\|_{L^2} + \|u^n\|_{L^4} \|\nabla c^n\|_{L^4} + \|\Delta \eta^{n+1}\|_{L^2} \right) \\
 &\quad + C \|P(\eta^n)\|_{L^\infty} \big( \|\eta^{n+1}\|_{L^2} + \|\Delta \eta^{n+1}\|_{L^2} + \|c^{n+1}\|_{L^2} \notag\\
 & \quad+ \|\Delta \eta^n\|_{L^2} + \|c^{n+1}\|_{L^2} + \|\eta^n\|_{L^2} \big)\notag \\
 \leq & C.\notag
 \end{align*}

%
%

\section{Numerical experiments}

\subsection*{Example 1. Single Tumor Growth}
We set the final time $T=0.2$, $\epsilon=0.02$, $\chi_0=0.02$,
$\delta=0.4$, $P_0=20$, $\kappa=0.25$, $\rho=1$, and $\alpha_\u=1$ for
the simulation. The auxiliary parameters are chosen as
$\lambda=4\delta\chi_0^2$ and $C_1=1$. The initial condition for a circular tumor
with sufficient nutrients is considered on the unit square domain.
\begin{equation}
	\begin{aligned}
		\eta_0(x,y)
		&=\frac{1}{2}\left[
		\tanh\left(
		\frac{0.15-\sqrt{(x-0.5)^2+(y-0.5)^2}}{\sqrt{2}\epsilon}
		\right)+1\right],\\
		c_0(x,y)&=1-\eta_0(x,y),
	\end{aligned}
	\label{eq:single_tumor_initial_eta_c}
\end{equation}
with the initial velocity and pressure given by
\begin{equation}
	\begin{aligned}
		\mathbf{u}_0(x,y)
		&=\left(
		-\frac{1}{4}\sin^2(x)\sin(2y),
		\frac{1}{4}\sin(2x)\sin^2(y)
		\right),\\
		p_0(x,y)&=\cos\left(\frac{x}{2}\right)\cos\left(\frac{y}{2}\right),
	\end{aligned}
	\label{eq:single_tumor_initial_up}
\end{equation}
where the pressure is projected to have zero mean.

To verify the convergence behavior, we use the same single-tumor-growth
example for both the time and space convergence tests. The
time-convergence results are computed with the reference solution on a $256\times256$ mesh and $N=2048$ time
steps. The test pairs (mesh resolution, number of time steps) are
$(16,16)$, $(32,32)$, $(64,64)$, $(128,128)$, and $(256,256)$.
The space-convergence results are computed with the reference solution on a $128\times128$ mesh and $N=4096$ time
steps. The test uses mesh resolutions $M\in\{4,8,16,32,64\}$ with
$N=M^2$ time steps.

\begin{figure}[!htbp]
	\centering
	\begin{subfigure}{0.47\textwidth}
		\centering
		\includegraphics[width=\textwidth]{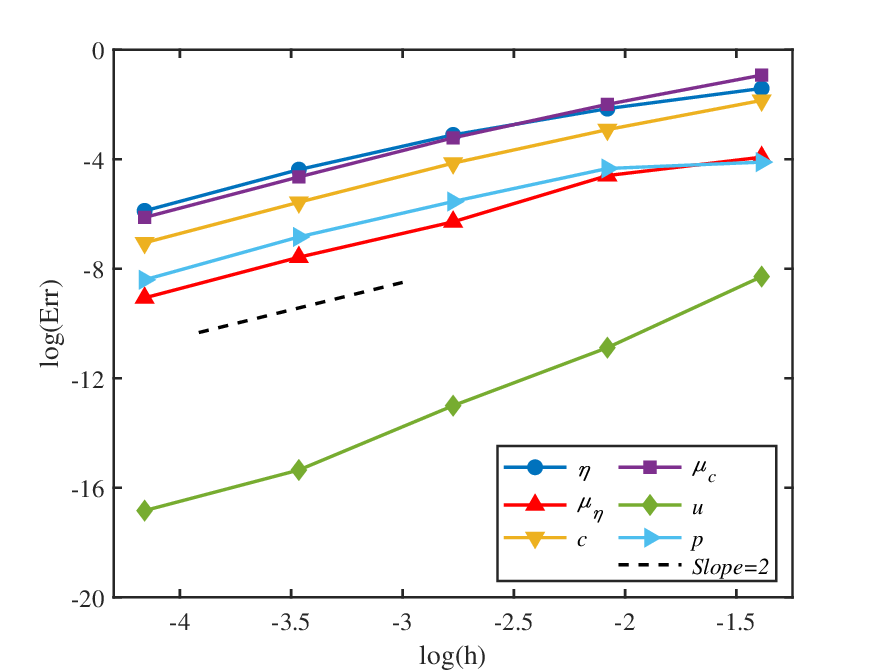}
		\caption{Spatial convergence results.}
		\label{fig:space_convergence}
	\end{subfigure}
	\hfill
	\begin{subfigure}{0.47\textwidth}
		\centering
		\includegraphics[width=\textwidth]{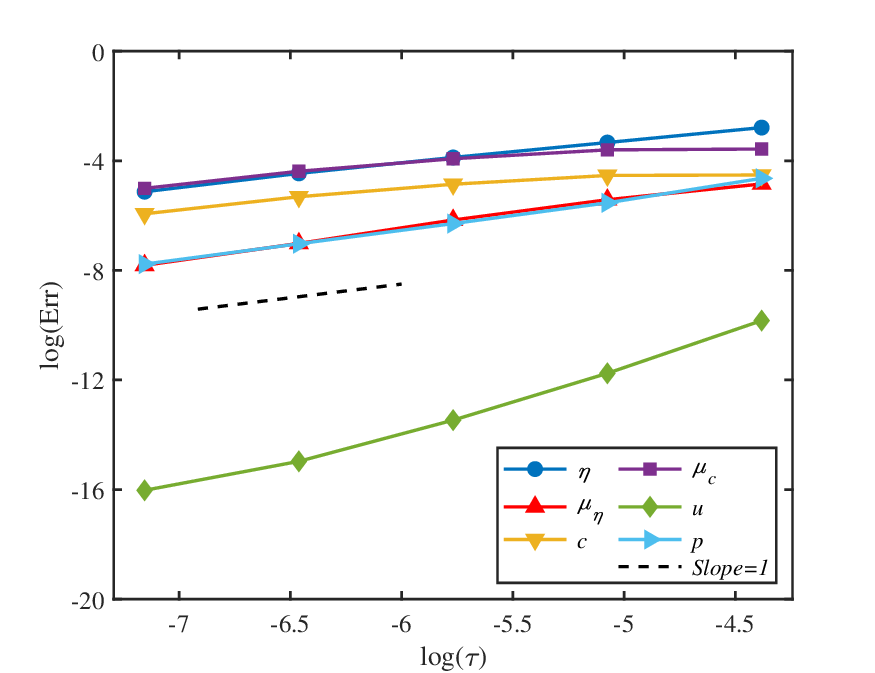}
		\caption{Temporal convergence results.}
		\label{fig:time_convergence}
	\end{subfigure}
	\caption{Convergence results.}
	\label{fig:convergence_results}
\end{figure}

Figure~\ref{fig:convergence_results} shows the spatial and temporal
convergence behavior of the scheme; the errors decrease as the mesh and
time step are refined.
The energy-monitoring test is performed for the same single-tumor growth example. A $96\times96$ mesh is employed with $N=20$ and $N=20{,}000$, corresponding to time step sizes $\Delta t=10^{-2}$ and $\Delta t=10^{-5}$, respectively. Figure~\ref{fig:mass_vs_time} presents the temporal evolution of the original energy together with the modified energy that includes the contribution of the auxiliary variable $q$. To further illustrate the consistency between the two energies, Figure~\ref{fig:energy_vs_time} displays the original energy and the modified energy after removing the $q$-dependent contribution. It can be observed that, as the time step decreases, the modified energy converges to the original energy. Therefore, the dissipation of the modified energy accurately reflects the dissipation of the original energy when a sufficiently small time step is adopted. 
Figure \ref{fig:mass_vs_time_single}
indicate that the total mass is conserved during the tumor growth, while the energy satisfies the expected dissipation behavior.

\begin{figure}[!htbp]
	\centering
	\begin{subfigure}{0.47\textwidth}
		\centering
		\includegraphics[width=\textwidth]{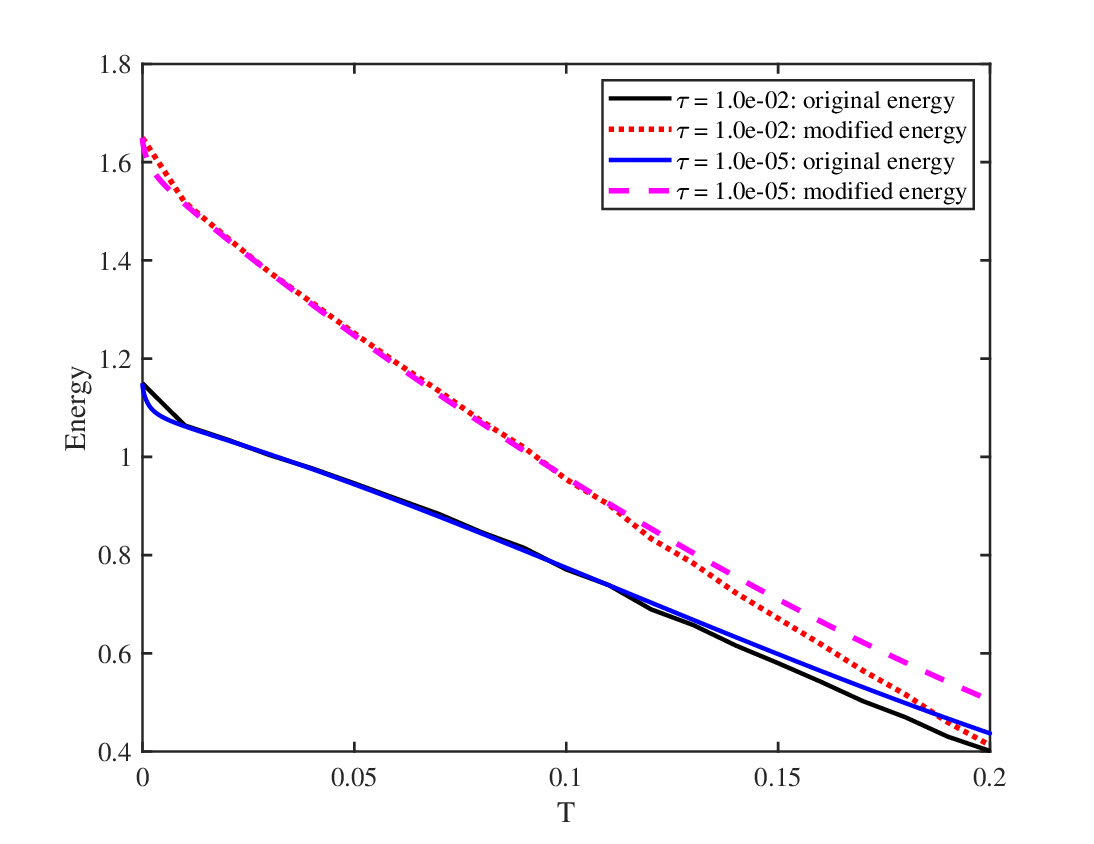}
		\caption{Total energy versus time with q.}
		\label{fig:mass_vs_time}
	\end{subfigure}
	\hfill
	\begin{subfigure}{0.47\textwidth}
		\centering
		\includegraphics[width=\textwidth]{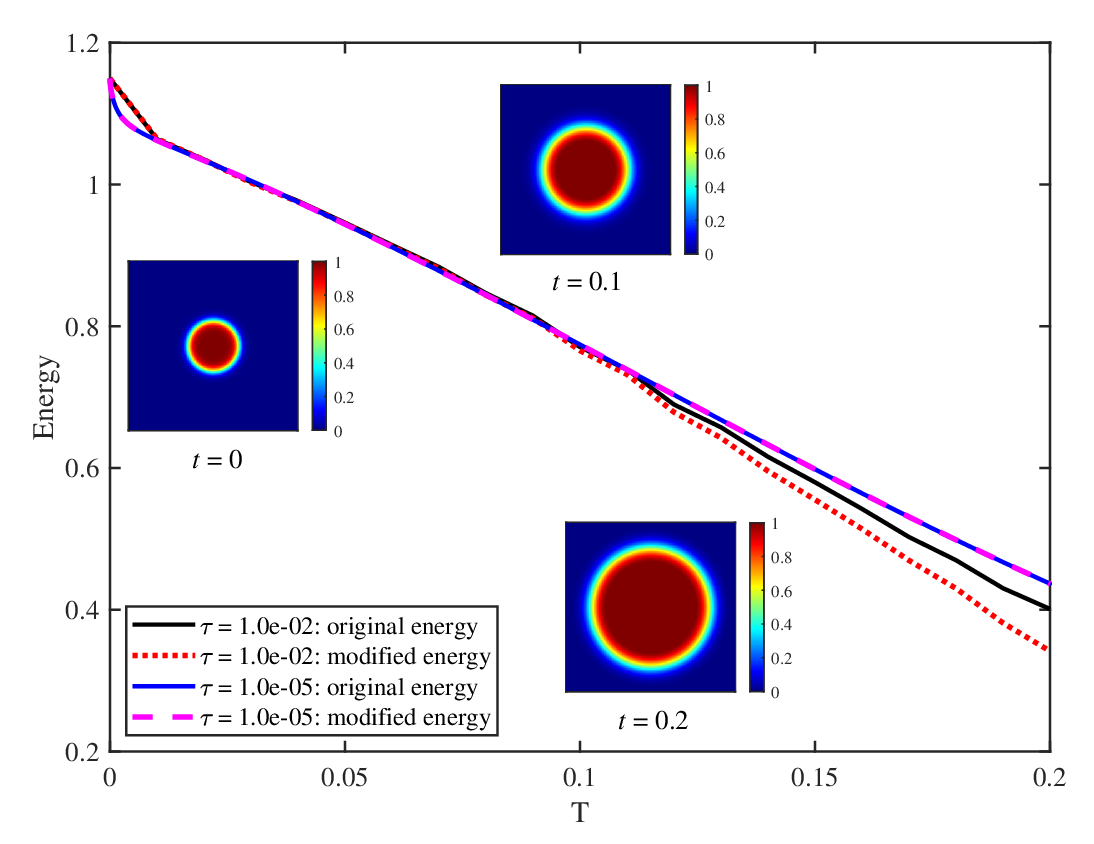}
		\caption{Total energy versus time without q.}
		\label{fig:energy_vs_time}
	\end{subfigure}
	\caption{Energy monitoring results.}
	\label{fig:mass_energy_monitoring}
\end{figure}

\begin{figure}[!htbp]
	\centering
	\includegraphics[width=0.4\textwidth]{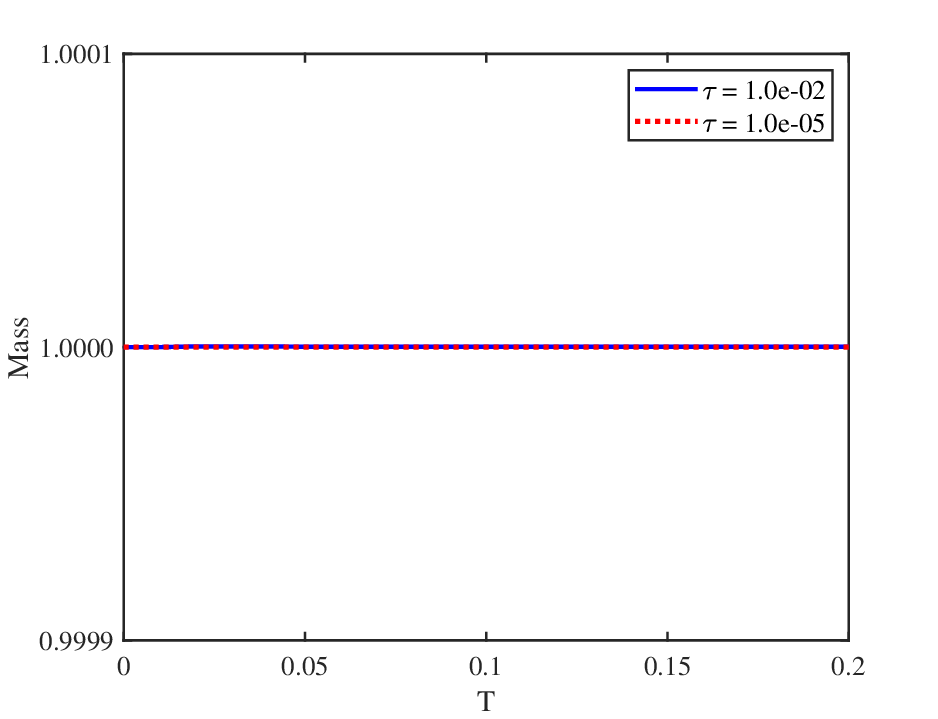}
	\caption{Total mass versus time.}
	\label{fig:mass_vs_time_single}
\end{figure}

\subsection*{Example 2. Two-Tumor Growth with Slower Nutrient Diffusion}
We consider a two-tumor-growth simulation on
$\Omega=[-1,1]^2$. The final time is $T=0.2$, and a $128\times128$ mesh
with $N=200$ time steps, corresponding to $\Delta t=10^{-3}$, is used.
The model and auxiliary parameters are
$
\epsilon=0.02,\quad
\chi_0=0.02,\quad
\delta=0.4,\quad
\kappa=0.25,\quad
P_0=20,\quad
\rho=1,\quad \alpha_\u=1,\quad
\lambda=4\delta\chi_0^2,\quad
C_1=1.
$
Thus the effective nutrient diffusion coefficient is $1/\delta=2.5$.
In the weak reaction term, the factor is
$P_n=P_0\delta\eta_n$ for $0\leq\eta_n\leq1$, and it is set to zero
otherwise. With the above parameters, this gives $P_n=8\eta_n$ within
the admissible phase interval. The solution snapshots are recorded at
$t=0$, $0.1$, and $0.2$.
The initial tumor and nutrient phases are defined as
\begin{equation}
	\eta_0(x,y)
	=1+\frac{1}{2}\sum_{\xi\in\{-0.3,0.3\}}
	\tanh\left(
	\frac{0.2-\sqrt{(x-\xi)^2+y^2}}{\sqrt{2}\,(0.02)}+1
	\right),
	\qquad
	c_0(x,y)=1-\eta_0(x,y).
	\label{eq:example2_two_tumor_initial}
\end{equation}
The initial velocity and pressure are chosen as
\begin{equation}
	\begin{aligned}
		\mathbf{u}_0(x,y)
		&=\left(
		-\frac{1}{4}\sin^2(x)\sin(2y),
		\frac{1}{4}\sin(2x)\sin^2(y)
		\right),\\
		p_0(x,y)&=\cos\left(\frac{x}{2}\right)
		\cos\left(\frac{y}{2}\right).
	\end{aligned}
	\label{eq:example2_two_tumor_initial_up}
\end{equation}

\begin{figure}[!htbp]
	\centering
	\includegraphics[width=\textwidth]{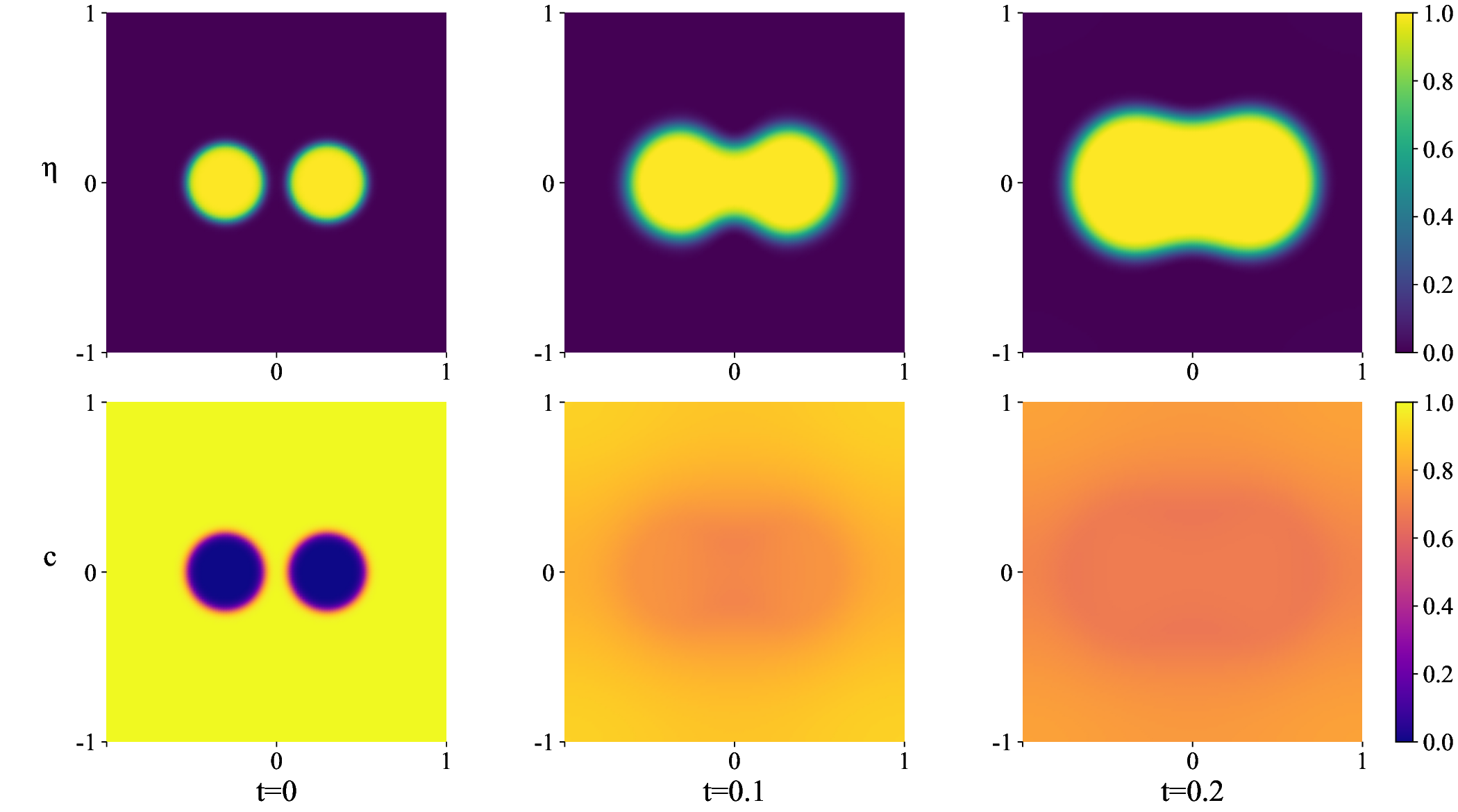}
	\caption{The evolution of the tumor phase $\eta$, and the nutrient concentration $c$ at $t=0,0.1,0.2$.}
	\label{fig:example2_multi_tumor_overview}
\end{figure}

\begin{figure}[!htbp]
	\centering
	\begin{subfigure}{0.47\textwidth}
		\centering
		\includegraphics[width=\textwidth]{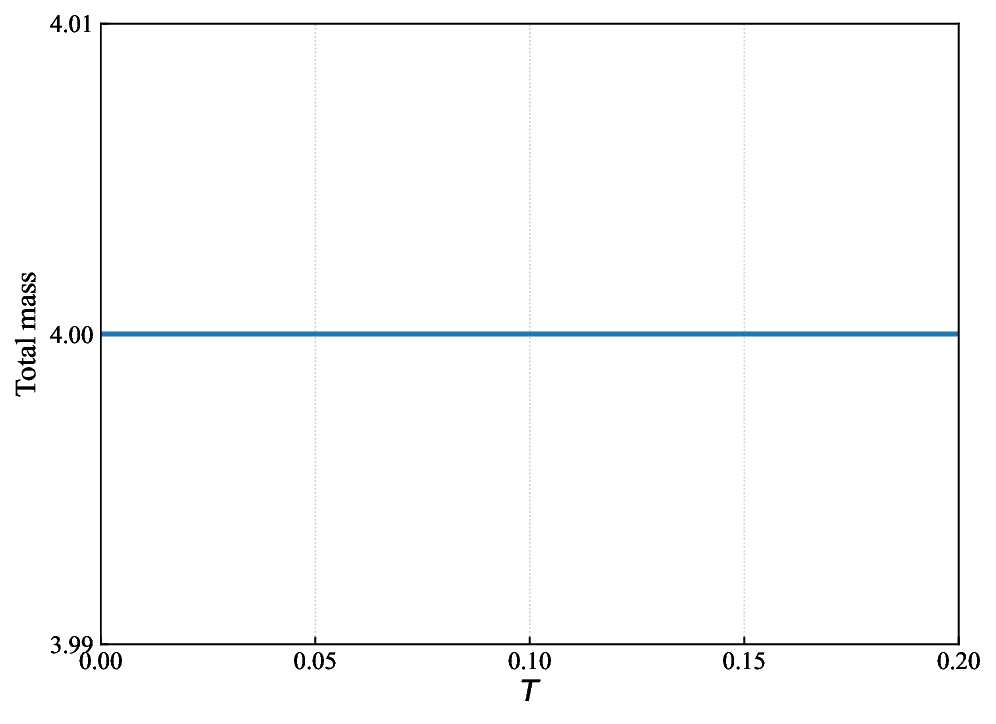}
		\caption{Total mass versus time.}
		\label{fig:example2_multi_tumor_mass}
	\end{subfigure}
	\hfill
	\begin{subfigure}{0.47\textwidth}
		\centering
		\includegraphics[width=\textwidth]{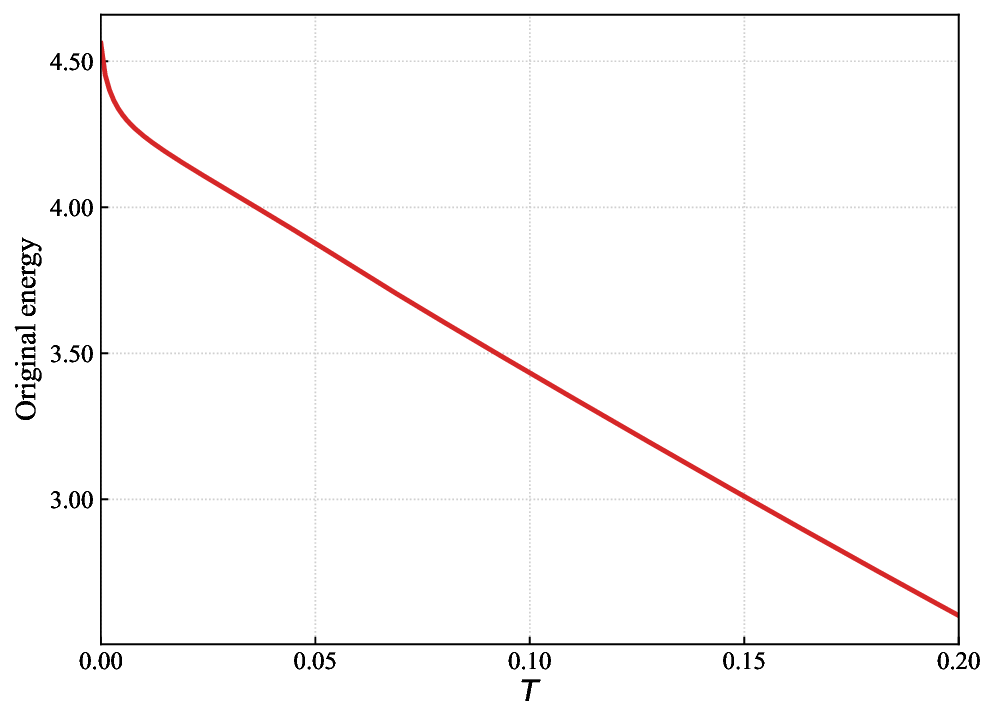}
		\caption{Original energy versus time.}
		\label{fig:example2_multi_tumor_energy}
	\end{subfigure}
	\caption{Mass and original-energy evolution for Example 2.}
	\label{fig:example2_multi_tumor_mass_energy}
\end{figure}

Figure~\ref{fig:example2_multi_tumor_overview} shows that the two
initial tumor regions merge into a connected tumor body while the
nutrient field is gradually depleted and smoothed. 
Figure~\ref{fig:example2_multi_tumor_mass_energy} shows that the total
mass remains essentially constant. The original energy decreases,
consistent with the
dissipative trend of the tumor-growth system.

\subsection*{Example 3. Chemotactic Tumor Growth with and without Fluid Coupling}
We next consider a chemotactic tumor-growth simulation on the square
domain $\Omega=[-1,1]^2$. The final time is $T=1$, and a
$128\times128$ mesh with
$\Delta t=5\times10^{-3}$. The model parameters are chosen as
\[
\epsilon=0.0283,\quad
\chi_0=0.112,\quad
\delta=10,\quad
\kappa=0.1,\quad
P_0=30,\quad
\lambda=0.04,\quad
C_1=1.
\] 
The scalar parameters and initial data described below are shared by
the cases without and with fluid coupling.

The initial tumor is an elliptic phase-field profile centered at the
origin defined as
\begin{equation}
 {\eta}_0(x,y)
=\frac{1}{2}\left[
1+\tanh\left(
\frac{1-\sqrt{(x/0.27)^2+(y/0.18)^2}}{\sqrt{2}\,(0.1)}+1
\right)\right].
\label{eq:example3_raw_eta_initial}
\end{equation}  
The initial nutrient field contains four smooth source regions. Denote
the source centers by
\[
\mathbf{x}_1=(-0.35,-0.35),\quad
\mathbf{x}_2=(0.35,0.35),\quad
\mathbf{x}_3=(0.35,-0.35),\quad
\mathbf{x}_4=(-0.35,0.35).
\]
For each source, set
\[
d_j(x,y)=\sqrt{(x-x_j)^2+(y-y_j)^2},
\qquad
r_c=0.2,\quad \varepsilon_c=0.05,
\]
and define
\begin{equation}
S_j(x,y)
=\frac{1}{2}\left[
1+\tanh\left(
\frac{r_c-d_j(x,y)}{\sqrt{2}\,\varepsilon_c}+1
\right)\right],
\qquad
S(x,y)=\sum_{j=1}^4 S_j(x,y).
\label{eq:example3_source_profile}
\end{equation}
The nutrient initial condition is then chosen as
\begin{equation}
c_0(x,y)=0.05+(1-0.05)S(x,y)^2.
\label{eq:example3_c_initial}
\end{equation}


The growth function is taken as a piecewise linear
function of the tumor phase,
\begin{equation}
P(\eta)=
\begin{cases}
P_0\eta, & 0\leq \eta \leq 1,\\
0, & \text{otherwise},
\end{cases}
\label{eq:example3_growth_function}
\end{equation}
and the reaction factor used in the weak form is
$P_n=\delta P(\eta_n)$. This choice allows tumor proliferation inside
the admissible phase interval while switching off the growth contribution
once the computed phase value leaves $[0,1]$.

\subsubsection*{Case I. Without Fluid Coupling}
In the first case, the velocity and pressure remain disabled with
$\mathbf{u}=\mathbf{0}$ and $p=0$ throughout the simulation.

\begin{figure}[!htbp]
\centering
\includegraphics[width=\textwidth]{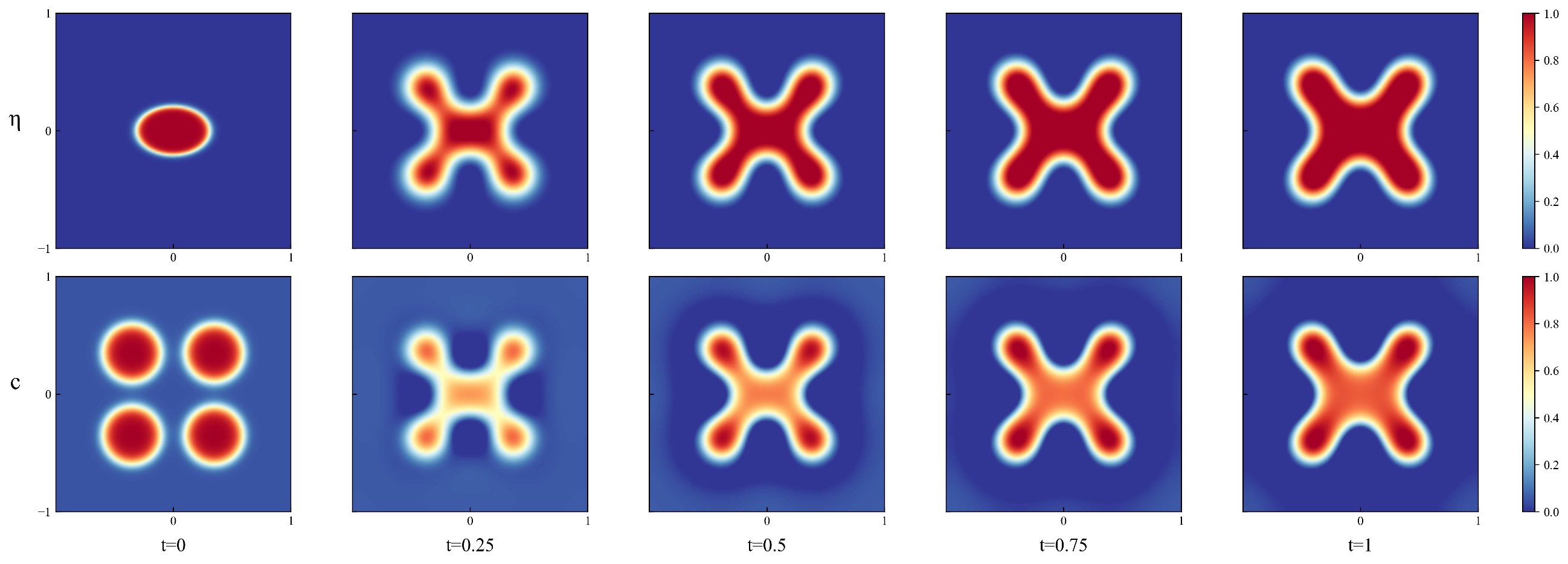}
\caption{Chemotactic tumor-growth evolution without fluid coupling.
The tumor phase $\eta$ and nutrient concentration $c$ are shown at
$t=0,0.25,0.5,0.75,1$.}
\label{fig:example3_no_fluid_overview}
\end{figure}

Figure~\ref{fig:example3_no_fluid_overview} shows that the initial
elliptic tumor is retained near the center while expanding toward the
four nutrient-rich regions. The chemotactic effect produces a four-lobed
tumor morphology, and the nutrient field is depleted and redistributed
along the same directions, which indicates that the growth is mainly
induced by the nutrient sources.

\begin{figure}[!htbp]
\centering
\includegraphics[width=0.4\textwidth]{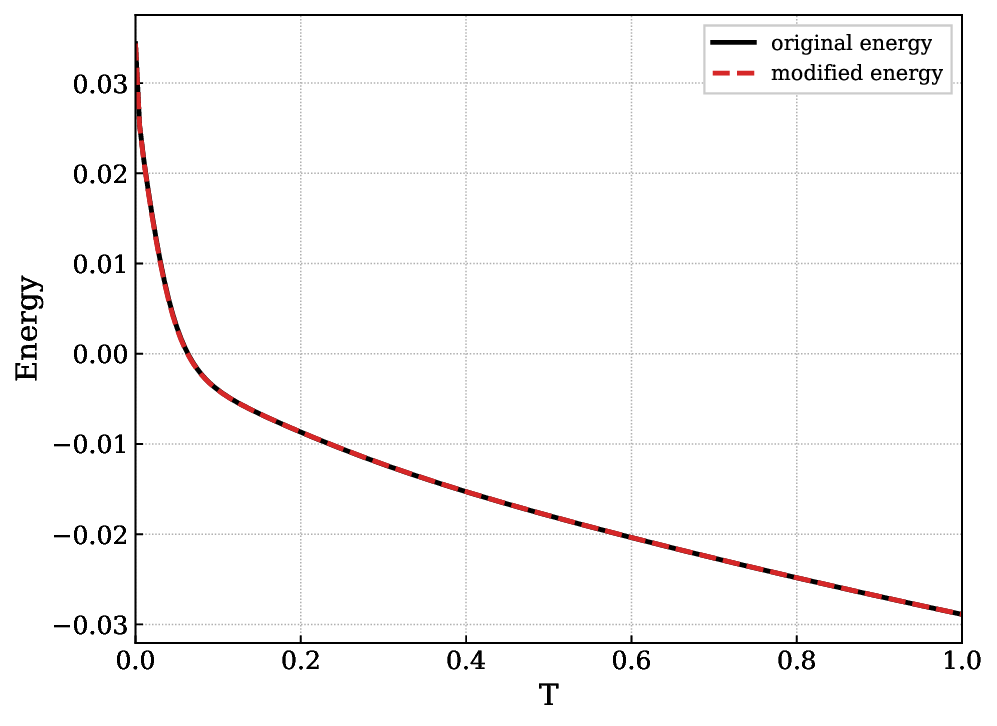}
\caption{Original and modified energies versus time without fluid
coupling.}
\label{fig:example3_no_fluid_energy}
\end{figure}

As shown in Figure~\ref{fig:example3_no_fluid_energy}, both recorded
energy curves exhibit the expected dissipative trend in the absence of
fluid coupling.

\subsubsection*{Case II. With Fluid Coupling}
In the second case, the full fluid-coupled system is used with
$\rho=1$ and $\alpha_\u=0.03$. The initial velocity is
\begin{equation}
\mathbf{u}_0(x,y)
=\left(
-25\sin^2(x)\sin(2y),
25\sin(2x)\sin^2(y)
\right),
\label{eq:example3_fluid_initial_u}
\end{equation}
and the initial pressure is obtained from
\begin{equation}
\widetilde{p}_0(x,y)
=\cos\left(\frac{x}{2}\right)\cos\left(\frac{y}{2}\right),
\qquad
p_0=\widetilde{p}_0
-\frac{1}{|\Omega|}\int_\Omega\widetilde{p}_0\,\mathrm{d}\mathbf{x}.
\label{eq:example3_fluid_initial_p}
\end{equation}

\begin{figure}[!htbp]
\centering
\includegraphics[width=\textwidth]{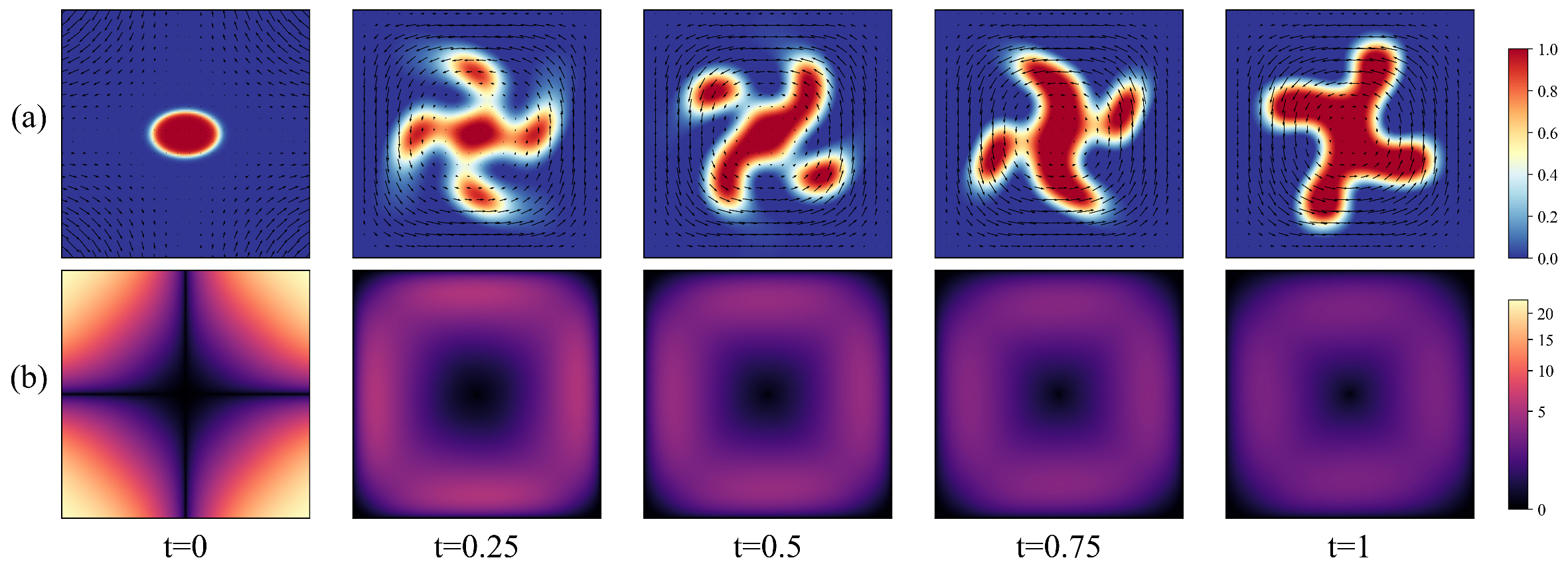}
\caption{Fluid-coupled evolution at
$t=0,0.25,0.5,0.75,1$: (a) tumor phase $\eta$ overlaid with velocity
vectors; (b) velocity magnitude $|\mathbf{u}|$. }
\label{fig:example3_with_fluid_overview}
\end{figure}

Figure~\ref{fig:example3_with_fluid_overview} shows that fluid advection
breaks the nearly symmetric growth pattern observed without fluid
coupling. The tumor branches are transported and rotated before merging
into an asymmetric connected structure. The initially strong flow
rapidly relaxes into a weaker rotating pattern while continuing to
transport the tumor interface.

\begin{figure}[!htbp]
\centering
\includegraphics[width=0.4\textwidth]{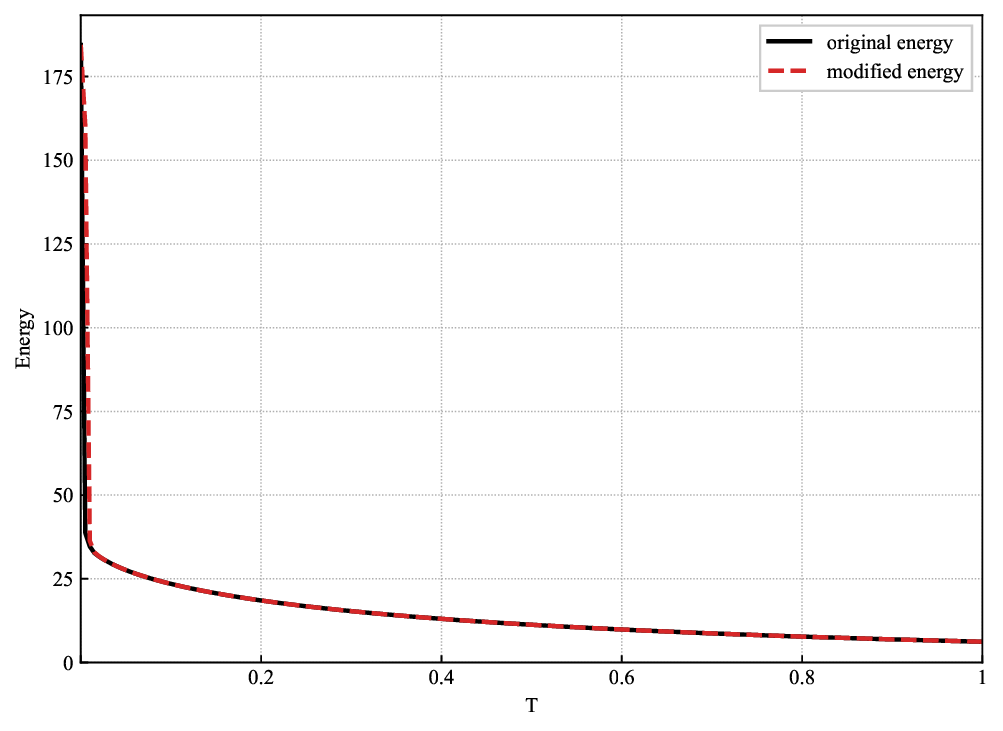}
\caption{Original and modified energies versus time with fluid
coupling.}
\label{fig:example3_with_fluid_energy}
\end{figure}

Figure~\ref{fig:example3_with_fluid_energy} shows that the original and
modified energies decrease throughout the fluid-coupled simulation.
The larger initial energy is due to the kinetic contribution of the
prescribed initial flow.

\begin{figure}[!htbp]
\centering
\includegraphics[width=0.4\textwidth]{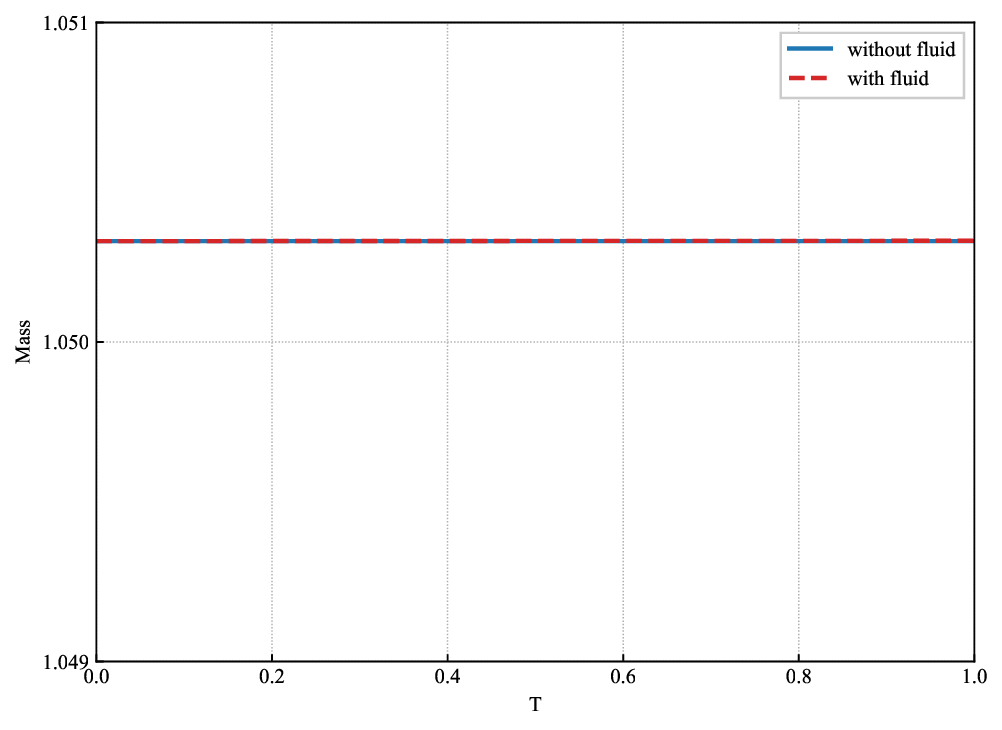}
\caption{Total mass versus time without and with fluid coupling.}
\label{fig:example3_mass_comparison}
\end{figure}

Finally, Figure~\ref{fig:example3_mass_comparison} compares the total
mass in the two cases. The mass remains essentially constant in both
simulations, while the energy results in
Figures~\ref{fig:example3_no_fluid_energy} and
\ref{fig:example3_with_fluid_energy} demonstrate the dissipative
behavior of the scheme in both regimes.

\section{Conclusion}
In this work, we developed a thermodynamically consistent phase-field model for tumor growth by coupling the Cahn–Hilliard and incompressible Navier–Stokes equations within the energetic variational approach, explicitly accounting for fluid dynamics and nutrient transport. To overcome the numerical challenges posed by strong nonlinearity, multiphysics coupling, and high-order operators, we proposed a structure-preserving numerical scheme based on the multiple scalar auxiliary variable framework combined with a projection–correction strategy for incompressibility. The resulting method is mass-conservative and rigorously inherits a discrete energy dissipation law, yielding unconditional energy stability. Moreover, under suitable regularity assumptions, optimal-order error estimates were established for the phase-field variable, nutrient concentration, and fluid velocity, and numerical experiments confirmed the theoretical convergence rates as well as the accuracy and robustness of the proposed approach. The present framework provides a solid foundation for future extensions to higher-order time discretizations, adaptive spatial refinement, and more complex biophysical tumor models.


\begin{thebibliography}{99}


\bibitem{ferlay2018}
J. Ferlay, M. Colombet, I. Soerjomataram, et al.,
\newblock Cancer incidence and mortality patterns in Europe:
estimates for 40 countries and 25 major cancers in 2018,
\newblock \emph{Eur. J. Cancer}, 103 (2018), pp.~356--387.

\bibitem{enderling2014}
H. Enderling and M. A. J. Chaplain,
\newblock Mathematical modeling of tumor growth and treatment,
\newblock \emph{Curr. Pharm. Des.}, 20 (30) (2014),
pp.~4934--4940.

\bibitem{roose2007}
T. Roose, S. J. Chapman, and P. K. Maini,
\newblock Mathematical models of avascular tumor growth,
\newblock \emph{SIAM Rev.}, 49 (2007), pp.~179--208.

\bibitem{xu2020}
J. Xu, G. Vilanova, and H. Gomez,
\newblock Phase-field model of vascular tumor growth:
three-dimensional geometry of the vascular network and integration
with imaging data,
\newblock \emph{Comput. Methods Appl. Mech. Engrg.},
359 (2020), 112648.

\bibitem{ebenbeck2019}
M. Ebenbeck and H. Garcke,
\newblock Analysis of a Cahn--Hilliard--Brinkman model for tumour
growth with chemotaxis,
\newblock \emph{J. Differential Equations},
266 (2019), pp.~5998--6036.

\bibitem{jiang2015}
J. Jiang, H. Wu, and S. Zheng,
\newblock Well-posedness and long-time behavior of a non-autonomous
Cahn--Hilliard--Darcy system with mass source modeling tumor growth,
\newblock \emph{J. Differential Equations},
259 (2015), pp.~3032--3077.

\bibitem{garcke2018}
H. Garcke, K. F. Lam, R. N\"urnberg, and E. Sitka,
\newblock A multiphase Cahn--Hilliard--Darcy model for tumour growth
with necrosis,
\newblock \emph{Math. Models Methods Appl. Sci.},
28 (2018), pp.~525--577.

\bibitem{oden2010}
J. T. Oden, A. Hawkins, and S. Prudhomme,
\newblock General diffuse-interface theories and an approach to
predictive tumor growth modeling,
\newblock \emph{Math. Models Methods Appl. Sci.},
20 (2010), pp.~477--517.

\bibitem{hilhorst2015}
D. Hilhorst, J. Kampmann, T. N. Nguyen, and K. G. van der Zee,
\newblock Formal asymptotic limit of a diffuse-interface tumor-growth
model,
\newblock \emph{Math. Models Methods Appl. Sci.},
25 (2015), pp.~1011--1043.


\bibitem{garcke2020}
H. Garcke and S. Yayla,
\newblock Long-time dynamics for a Cahn--Hilliard tumor growth model
with chemotaxis,
\newblock \emph{Z. Angew. Math. Phys.},
71 (2020), Article~123, 32~pp.

\bibitem{riva2025}
F. Riva and E. Rocca,
\newblock A rigorous approach to the sharp interface limit for
phase-field models of tumor growth,
\newblock \emph{SIAM J. Math. Anal.},
57 (1) (2025), pp.~65--94.

\bibitem{colli2015}
P. Colli, G. Gilardi, and D. Hilhorst,
\newblock On a Cahn--Hilliard type phase field system related to
tumor growth,
\newblock \emph{Discrete Contin. Dyn. Syst.},
35 (6) (2015), pp.~2423--2442.

\bibitem{colli20152}
P. Colli, G. Gilardi, E. Rocca, and J. Sprekels,
\newblock Vanishing viscosities and error estimate for a
Cahn--Hilliard type phase field system related to tumor growth,
\newblock \emph{Nonlinear Anal. Real World Appl.},
26 (2015), pp.~93--108.

\bibitem{colli2017}
P. Colli, G. Gilardi, E. Rocca, and J. Sprekels,
\newblock Asymptotic analyses and error estimates for a
Cahn--Hilliard type phase field system modelling tumor growth,
\newblock \emph{Discrete Contin. Dyn. Syst. Ser. S},
10 (1) (2017), pp.~37--54.

\bibitem{frigeri2015}
S. Frigeri, M. Grasselli, and E. Rocca,
\newblock On a diffuse interface model of tumour growth,
\newblock \emph{European J. Appl. Math.},
26 (2) (2015), pp.~215--243.


\bibitem{shen2023}
X. Shen, L. Wu, J. Wen, and J. Zhang,
\newblock SAV Fourier-spectral method for diffuse-interface
tumor-growth model,
\newblock \emph{Comput. Math. Appl.},
140 (2023), pp.~250--259.

\bibitem{zou2022}
G. Zou, B. Wang, and X. Yang,
\newblock A fully-decoupled discontinuous Galerkin approximation and
optimal error estimate of the
Cahn--Hilliard--Brinkman--Ohta--Kawasaki tumor growth model,
\newblock \emph{ESAIM Math. Model. Numer. Anal.},
56 (2022), pp.~2141--2180.

\bibitem{wu2014}
X. Wu, G. J. van Zwieten, and K. G. van der Zee,
\newblock Stabilized second-order convex splitting schemes for
Cahn--Hilliard models with application to diffuse-interface
tumor-growth models,
\newblock \emph{Int. J. Numer. Methods Biomed. Eng.},
30 (2014), pp.~180--203.

\bibitem{hawkins2012}
A. Hawkins-Daarud, K. G. van der Zee, and J. T. Oden,
\newblock Numerical simulation of a thermodynamically consistent
four-species tumor growth model,
\newblock \emph{Int. J. Numer. Methods Biomed. Eng.},
28 (1) (2012), pp.~3--24.

\bibitem{lin2025}
Z. Wang, P. Lin, and J. Yang,
\newblock Stability and error analysis of structure-preserving
schemes for a diffuse-interface tumor growth model,
\newblock \emph{SIAM J. Sci. Comput.},
47 (1) (2025), pp.~B59--B86.


\bibitem{ribba2006}
B. Ribba, O. Saut, T. Colin, D. Bresch, E. Grenier, and J. P. Boissel,
\newblock A multiscale mathematical model of avascular tumor growth
to investigate the therapeutic benefit of anti-invasive agents,
\newblock \emph{J. Theoret. Biol.},
243 (2006), pp.~532--541.

\bibitem{moha2019}
V. Mohammadi and M. Dehghan,
\newblock Simulation of the phase field Cahn--Hilliard and tumor
growth models via a numerical scheme: element-free Galerkin method,
\newblock \emph{Comput. Methods Appl. Mech. Engrg.},
345 (2019), pp.~919--950.

\bibitem{netti2000}
P. A. Netti, D. A. Berk, M. A. Swartz, A. J. Grodzinsky, and
R. K. Jain,
\newblock Role of extracellular matrix assembly in interstitial
transport in solid tumors,
\newblock \emph{Cancer Res.},
60 (2000), pp.~2497--2503.


\bibitem{evans1949}
L. C. Evans,
\newblock \emph{Partial Differential Equations},
\newblock 2nd ed., Graduate Studies in Mathematics, Vol.~19,
American Mathematical Society, Providence, RI, 2010.

\bibitem{temam2024}
R. Temam,
\newblock \emph{Navier--Stokes Equations:
Theory and Numerical Analysis},
\newblock AMS Chelsea Publishing, Vol.~343,
American Mathematical Society, Providence, RI, 2024.


\bibitem{lin2022}
L. Shen, Z. Xu, P. Lin, H. Huang, and S. Xu,
\newblock An energy-stable \(C^0\) finite element scheme for a
phase-field model of vesicle motion and deformation,
\newblock \emph{SIAM J. Sci. Comput.},
44 (1) (2022), pp.~B122--B145.

%


\bibitem{shenjie2018}
J. Shen and J. Xu,
\newblock Convergence and error analysis for the scalar auxiliary
variable (SAV) schemes to gradient flows,
\newblock \emph{SIAM J. Numer. Anal.},
56 (5) (2018), pp.~2895--2912.

\bibitem{shenjie2021}
X. Li, J. Shen, and Z. Liu,
\newblock New SAV-pressure correction methods for the
Navier--Stokes equations: stability and error analysis,
\newblock \emph{Math. Comp.},
91 (333) (2022), pp.~141--167.

\bibitem{shenjie1992}
J. Shen,
\newblock On error estimates of projection methods for
Navier--Stokes equations: first-order schemes,
\newblock \emph{SIAM J. Numer. Anal.},
29 (1992), pp.~57--77.

\bibitem{cheng2018}
Q. Cheng and J. Shen,
\newblock Multiple scalar auxiliary variable (MSAV) approach and its
application to the phase-field vesicle membrane model,
\newblock \emph{SIAM J. Sci. Comput.},
40 (6) (2018), pp.~A3982--A4006.

\bibitem{yang2017}
X. Yang and L. Ju,
\newblock Linear and unconditionally energy-stable schemes for the
binary fluid--surfactant phase-field model,
\newblock \emph{Comput. Methods Appl. Mech. Engrg.},
318 (2017), pp.~1005--1029.

\bibitem{lin2011}
J. Hua, P. Lin, C. Liu, and Q. Wang,
\newblock Energy law preserving \(C^0\) finite element schemes for
phase-field models in two-phase flow computations,
\newblock \emph{J. Comput. Phys.},
230 (19) (2011), pp.~7115--7131.

\bibitem{lin2014}
Z. Guo, P. Lin, and J. S. Lowengrub,
\newblock A numerical method for the quasi-incompressible
Cahn--Hilliard--Navier--Stokes equations for variable-density flows
with a discrete energy law,
\newblock \emph{J. Comput. Phys.},
276 (2014), pp.~486--507.

\bibitem{shenjie2010}
J. Shen and X. Yang,
\newblock Numerical approximations of Allen--Cahn and
Cahn--Hilliard equations,
\newblock \emph{Discrete Contin. Dyn. Syst.},
28 (2010), pp.~1669--1691.

\bibitem{heywood1990}
J. G. Heywood and R. Rannacher,
\newblock Finite-element approximation of the nonstationary
Navier--Stokes problem. IV. Error analysis for second-order time
discretization,
\newblock \emph{SIAM J. Numer. Anal.},
27 (2) (1990), pp.~353--384.

\bibitem{shenjie2022}
X. Li and J. Shen,
\newblock On fully decoupled MSAV schemes for the
Cahn--Hilliard--Navier--Stokes model of two-phase incompressible
flows,
\newblock \emph{Math. Models Methods Appl. Sci.},
32 (3) (2022), pp.~457--495.

\end{thebibliography}
\end{document}